\documentclass[11pt,a4paper]{article}

\usepackage{a4wide}
\usepackage{amsmath}
\usepackage{amssymb}
\usepackage{mathtools}
\usepackage{bm}
\mathtoolsset{showonlyrefs=true}
\usepackage{booktabs}
\usepackage{algorithm}
\usepackage{algpseudocode}
\usepackage[utf8]{inputenc}
\usepackage[T1]{fontenc}
\usepackage[numbers,sort&compress]{natbib}
\usepackage{hyperref}
\usepackage{amsfonts}
\usepackage{graphicx}
\usepackage{subcaption}
\usepackage{cleveref}
\usepackage{amsthm}
\usepackage{url}
\usepackage{mathrsfs}
\usepackage{nicefrac}
\usepackage{microtype}
\usepackage{fancyhdr}

\DeclareMathOperator{\diag}{diag}

\DeclareMathOperator{\grad}{grad}

\DeclareMathOperator{\Hess}{Hess}
\DeclareMathOperator{\inv}{inv}

\DeclareMathOperator{\Rank}{rank}
\DeclareMathOperator{\sgn}{sgn}

\let\skew\relax
\DeclareMathOperator{\skew}{skew}
\DeclareMathOperator{\Skew}{Skew}
\DeclareMathOperator{\Span}{span}

\DeclareMathOperator{\St}{St}
\DeclareMathOperator{\iSt}{iSt}

\DeclareMathOperator{\sym}{sym}
\DeclareMathOperator{\Sym}{Sym}
\DeclareMathOperator{\tr}{tr}

\def\D{\mathrm{D}}
\def\id{\mathrm{id}}

\def\E{\mathcal{E}}
\def\R{\mathbb{R}}
\def\M{\mathcal{M}}
\def\N{\mathcal{N}}
\def\T{\mathscr{T}}

\newtheorem{proposition}{Proposition}[section]

\newtheorem{lemma}{Lemma}[section]

\theoremstyle{remark}

\newtheorem{remark}{Remark}[section]

\crefname{proposition}{Proposition}{Propositions}
\crefname{theorem}{Theorem}{Theorems}
\crefname{corollary}{Corollary}{Corollaries}
\crefname{lemma}{Lemma}{Lemmas}
\crefname{definition}{Definition}{Definitions}
\crefname{remark}{Remark}{Remarks}
\crefname{assumption}{Assumption}{Assumptions}
\crefname{example}{Example}{Examples}
\crefname{problem}{Problem}{Problems}
\crefname{section}{Section}{Sections}
\crefname{subsection}{Section}{Sections}
\crefname{subsubsection}{Section}{Sections}
\crefname{figure}{Figure}{Figures}

\newenvironment{keywords}{\par\noindent\textbf{Keywords:}\ }{\par}
\newcommand{\sep}{\unskip,\ }

\title{Second-order geometry and Riemannian Newton-type methods\\
for optimization on the indefinite Stiefel manifold}
\author{Hiroyuki Sato\thanks{Department of Mathematical Sciences, Ritsumeikan University, 1-1-1 Noji-higashi, Kusatsu-shi, Shiga 525-8577, Japan. Corresponding author: \texttt{hsato@fc.ritsumei.ac.jp}. ORCID: \href{https://orcid.org/0000-0003-1399-8140}{0000-0003-1399-8140}.}}
\date{}

\begin{document}
\maketitle

\begin{abstract}
This paper investigates the second-order geometry of the indefinite Stiefel manifold and derives explicit formulas for the Levi-Civita connection and the Riemannian Hessian under two generalized canonical metrics. We discuss Riemannian Newton's method, in which Newton's equation is solved by the linear conjugate gradient method in a fixed tangent space, and the Riemannian trust-region method with the truncated conjugate gradient method. Numerical experiments for trace minimization problems demonstrate the robustness of the trust-region method over several problem sizes and in near-singular settings where eigenvalues of the constraint matrix approach zero.
\end{abstract}


\section{Introduction}
Riemannian optimization, which is a generalization of continuous optimization in the Euclidean space to Riemannian manifolds, has been extensively studied~\cite{AbsMahSep2008,boumal2023intromanifolds,edelman1998geometry,sato2021riemannian,smith1994optimization}.
An important example of manifolds on which we study optimization is the Stiefel manifold
\begin{equation}
\St(p,n) \coloneqq \{X \in \R^{n \times p} \mid X^{\top}X = I_p\},
\end{equation}
where $p, n \in \mathbb{N}$ satisfy $p \leq n$.
Each point $X = \begin{bmatrix}
x_1 & x_2 & \cdots &x_p
\end{bmatrix} \in \St(p,n)$ ($x_1, x_2, \dots, x_p \in \R^n$) can be regarded as an orthonormal $p$-frame in $\R^n$ with respect to the standard inner product $\R^n \times \R^n \to \R \colon (a, b) \mapsto \langle a, b\rangle \coloneqq a^{\top}b$ since we have $\langle x_i, x_j\rangle = x_i^{\top}x_j = \delta_{ij}$ from $X^{\top}X = I_p$, where $\delta_{ij}$ is Kronecker's delta.
If we endow $\R^n$ with a general inner product $\langle a, b\rangle_{B} \coloneqq a^{\top}Bb$, where $B$ is an $n \times n$ symmetric positive definite matrix, then the set of orthonormal $p$-frames in $\R^n$ with respect to this inner product constitutes the generalized Stiefel manifold~\cite{sato2019cholesky,yger2012adaptive}
\begin{equation}
\St_{B}(p,n) \coloneqq \{X \in \R^{n \times p} \mid X^{\top}BX = I_p\}.
\end{equation}

We can also endow $\R^n$ with an indefinite inner product $\langle a, b\rangle_A \coloneqq a^{\top}Ab$, where $A$ is an invertible and indefinite symmetric matrix, i.e., $A$ has both positive and negative eigenvalues.
The indefinite inner product is symmetric and nondegenerate, but does not possess the property of positive-definiteness.
We say that $a \in \R^n$ is a unit vector with respect to the indefinite inner product if $\langle a, a\rangle_A = \pm 1$ and that $a$ and $b \in \R^n$ are orthogonal if $\langle a, b\rangle_A = 0$.
Then, for a $p$-frame $X = \begin{bmatrix}
x_1 & x_2 & \cdots & x_p
\end{bmatrix}$, where each $x_i$ is a unit vector and $x_i$ and $x_j$ ($i \neq j$) are orthogonal, we have $\langle x_i, x_j\rangle_A = \pm \delta_{ij}$.
Here, the number of $i \in \{1, 2, \dots, p\}$ satisfying $\langle x_i, x_i\rangle_A = 1$ (resp. $-1$) is not greater than that of positive (resp. negative) eigenvalues of $A$.
Among such $X$, we consider the case $\langle x_i, x_i\rangle_A = 1$ for $1 \leq i \leq p_+$ and $\langle x_i, x_i\rangle_A = -1$ for $p_+ + 1 \leq i \leq p_+ + p_-$, where nonnegative integers $p_+$ and $p_-$ satisfy $p_+ + p_- = p$.
Then, we have $X^{\top}A X = \diag(I_{p_+}, -I_{p_-})$, where we note that the square of the right-hand side is equal to the identity matrix $I_p$.

Generalizing all the above cases, the indefinite Stiefel manifold~\cite{van2025generalized,van2024riemannian} is defined to be
\begin{equation}
\iSt_{A,J}(p,n) \coloneqq \{X \in \R^{n \times p} \mid X^{\top}AX = J\},
\end{equation}
where $A$ is an $n \times n$ invertible symmetric matrix and $J$ is a $p \times p$ symmetric matrix with $J^2 = I_p$, which means that each eigenvalue of $J$ is $1$ or $-1$.
For generality, despite the word ``indefinite,'' we allow $A$ to be either definite or indefinite.
Therefore, the Stiefel and generalized Stiefel manifolds are special cases of the indefinite Stiefel manifold as $\iSt_{I_n, I_p}(p,n) = \St(p,n)$ and $\iSt_{B, I_p}(p,n) = \St_B(p,n)$.
The example shown in the previous paragraph is $\iSt_{A, J}(p,n)$ with $J = \diag(I_{p_+}, -I_{p_-})$.

The steepest descent method on $\iSt_{A,J}(p,n)$ is proposed in~\cite{van2024riemannian}.
A subsequent study~\cite{van2025generalized} showed that suitably chosen generalized canonical metrics substantially reduce the computational burden.
While such studies have revealed the geometry of the indefinite Stiefel manifold to a certain extent, the second-order geometry of the manifold, which is essential for Newton's and trust-region methods, has not been studied, to the best of the author's knowledge.

This study first investigates the second-order geometry of the indefinite Stiefel manifold $\iSt_{A,J}(p,n)$.
Specifically, given the Riemannian metrics in~\cite{van2025generalized}, we derive a formula for the associated Levi-Civita connection by means of Koszul's formula.
Then, we can compute the Hessian of an objective function on $\iSt_{A,J}(p,n)$ by using the Levi-Civita connection.

As for applications, optimization with orthogonality constraints is important in various areas such as signal processing. Examples include principal component analysis and joint approximate diagonalization arising in independent component analysis.
These problems naturally lead to optimization over the Stiefel manifold~\cite{edelman1998geometry,sato2017riemannian}.
Regarding a more general (positive-definite) inner product in the Euclidean space, the canonical correlation analysis is formulated as an optimization problem on the product of two generalized Stiefel manifolds~\cite{sato2019cholesky,yger2012adaptive}.
As a more generalized version, the symplectic Stiefel manifold is also studied~\cite{gao2021riemannian,yamada2023conjugate}.
Beyond positive-definite inner products, optimization problems with quadratic constraints of the form $X^{\top}AX=J$, which are the main focus of this paper, arise in several data-analytic and signal-processing–related tasks, including symmetric generalized eigenvalue problems involving an indefinite constraint matrix.
This motivates studying optimization problems whose feasible set consists of frames that are orthonormal with respect to an indefinite inner product, leading to the indefinite Stiefel manifold.
Further details and additional examples of the indefinite Stiefel manifold can be found in~\cite{van2024riemannian,van2025generalized}.

The remainder of this paper is organized as follows.
In Section~\ref{sec:Notation}, we introduce our notation and state preliminary assumptions.
In Section~\ref{sec:Review}, we review the indefinite Stiefel manifold $\iSt_{A,J}(p,n)$ and first-order geometry on $\iSt_{A,J}(p,n)$, mainly following the papers~\cite{van2024riemannian,van2025generalized}.
We also prepare and prove some propositions used in later sections.
In Section~\ref{sec:second-order}, we derive the Levi-Civita connection and Hessian of a function on $\iSt_{A,J}(p,n)$.
To this end, we first derive the Levi-Civita connection on an ambient manifold, which is an open submanifold of $\R^{n \times p}$, by using Koszul's formula.
Section~\ref{sec:Newton} explains how the derived Hessian can be used in Newton-type methods.
Specifically, we discuss Riemannian Newton's method and trust-region method.
In Section~\ref{sec:NumericalExperiments}, we present numerical results for trace minimization problems on the indefinite Stiefel manifold.
Section~\ref{sec:conclusion} concludes the paper.

\section{Preliminaries}
\label{sec:Notation}
Let $\Sym(n)$ (resp. $\Sym_{++}(n)$) with $n \in \mathbb{N}$ denote the set of all $n \times n$ real symmetric (resp. symmetric positive-definite) matrices.
For a square matrix $S$, $\sym(S) \coloneqq (S+S^{\top})/2$ and $\skew(S) \coloneqq (S-S^{\top})/2$ denote the symmetric and skew-symmetric parts of $S$, respectively.

Throughout the paper, we fix $p, n \in \mathbb{N}$ satisfying $p \leq n$, invertible symmetric matrix $A \in \Sym(n)$, and $J \in \Sym(p)$ satisfying $J^2 = I_p$.
We assume that $J$ has $p_+$ positive and $p_-$ negative eigenvalues, where $p_+$ and $p_-$ are nonnegative integers not greater than the numbers of positive and negative eigenvalues of $A$, respectively.

Let $\E$ denote the set of all $n \times p$ full-rank (i.e., of rank $p$) matrices $X$ such that $X^{\top}AX$ is invertible.
Equivalently, $\E$ is the intersection of two open sets as
\begin{equation}
\label{eq:openE}
\E = \{X \in \R^{n \times p} \mid \det(X^{\top}X) \neq 0\} \cap \{X \in \R^{n \times p} \mid \det(X^{\top}AX) \neq 0\}.
\end{equation}
Therefore, $\E$ is an open set, and thus an open submanifold of $\R^{n \times p}$~\cite{lee2018introduction}.
{
In the next section, we regard the manifold $\E$ as an ambient space of the indefinite Stiefel manifold $\iSt_{A,J}(p,n)$.
In particular, we equip $\E$ with a Riemannian metric (see below for the definition) and $\iSt_{A,J}(p,n)$ with the induced metric, thereby regarding $\iSt_{A,J}(p,n)$ as a Riemannian submanifold of $\E$.
}

For a manifold $\M$ and a point $x \in \M$, the tangent space of $\M$ at $x$ is denoted by $T_x \M$.
The tangent bundle of $\M$ is denoted by $T\M$, which is the direct sum of the tangent spaces at all points on $\M$.
The derivative of a smooth (i.e., smooth at any point as a map between Euclidean spaces via local coordinates) map $\varphi\colon \M \to \N$ between two manifolds $\M$ and $\N$ is denoted by $\D \varphi(x)$, which is a map from $T_x \M$ to $T_{\varphi(x)}\N$ defined to satisfy
\begin{equation}
\D \varphi(x)[\xi] = \frac{d}{dt}\varphi(c(t))\bigg|_{t = 0}
\end{equation}
for any curve $c$ on $\M$ with $c(0) = x$ and $\dot{c}(0) = \xi \in T_x \M$.
A smooth vector field $U$ on $\M$ is a smooth map from $\M$ to $T\M$ that associates each point $x \in \M$ a tangent vector $U(x) \in T_x \M$.
We denote the sets of all smooth functions and all smooth vector fields on $\M$ by $\mathfrak{F}(\M)$ and $\mathfrak{X}(\M)$, respectively.
A Riemannian metric $\langle \cdot, \cdot\rangle \colon \M \ni x \mapsto \langle \cdot, \cdot\rangle_x$ on $\M$ is a family of inner products in tangent spaces of $\M$ that is smooth with respect to $x$, i.e., $\langle \cdot, \cdot\rangle_x$ is an inner product in the vector space $T_x \M$ and $x \mapsto \langle U(x), V(x)\rangle_x$ is a smooth function on $\M$ for any vector fields $U, V \in \mathfrak{X}(\M)$.
{A manifold endowed with a Riemannian metric is called a Riemannian manifold. For a smooth function $f \colon \M \to \R$, the Riemannian gradient $\grad f(x)$ at $x \in \M$ is defined to satisfy
\begin{equation}
\D f(x)[\xi] = \langle \grad f(x), \xi\rangle_x
\end{equation}
for any tangent vector $\xi \in T_x \M$~\cite{AbsMahSep2008,boumal2023intromanifolds}.
}

When we deal with second-order optimization methods on a Riemannian manifold $\M$, we need the Hessian of the smooth objective function $f \colon \M \to \R$.
To define the Hessian, the concept of the Levi-Civita connection, which is a special case of affine connections, is used.
A map $\nabla \colon \mathfrak{X}(\M) \times \mathfrak{X}(\M) \to \mathfrak{X}(\M) \colon (U, V) \mapsto \nabla_U V$ is called an affine connection if it satisfies, for any $U, V, W \in \mathfrak{X}(\M)$, $f, g \in \mathfrak{F}(\M)$, and $\alpha, \beta \in \R$,
1) $\nabla_{fU + gW} V = f\nabla_U V + g\nabla_W V$; 2) $\nabla_U(\alpha V + \beta W) = \alpha\nabla_U V + \beta\nabla_U W$; and 3) $\nabla_U (fV) = (Uf)V + f\nabla_U V$.
An affine connection $\nabla$ is called the Levi-Civita connection if it further satisfies 4) $[U, V] = \nabla_U V - \nabla_V U$ and 5) $U\langle V, W\rangle = \langle \nabla_U V, W\rangle + \langle V, \nabla_U W\rangle$.
Here, $fU \in \mathfrak{X}(\M)$, $Uf \in \mathfrak{F}(\M)$, $[U, V] \colon \mathfrak{F}(\M) \to \mathfrak{F}(\M)$, and $\langle V, W\rangle \in \mathfrak{F}(\M)$ are defined by
$(fU)(x) \coloneqq f(x)U(x)$, $(Uf)(x) \coloneqq \D f(x)[U(x)]$, $[U, V](f) \coloneqq U(Vf) - V(Uf)$, and $\langle V, W\rangle(x) \coloneqq \langle V(x), W(x)\rangle_x$, respectively.
Since $(\nabla_U V)(x)$ depends on $U$ only through the value $U(x) \eqqcolon \xi$ at $x$, we sometimes simply denote $(\nabla_U V)(x)$ as $\nabla_{\xi}V$.
Then, the Hessian $\Hess f(x) \colon T_x \M \to T_x \M$ of $f \in \mathfrak{F}(\M)$ at $x \in \M$ is defined to satisfy
\begin{equation}
\Hess f(x)[\xi] = \nabla_{\xi}\grad f
\end{equation}
for any $\xi \in T_x \M$, where $\grad f \in \mathfrak{X}(\M)$ is the Riemannian gradient vector field of $f$ that maps $x \mapsto \grad f(x)$.

\section{Review of the indefinite Stiefel manifold and some additional propositions}
\label{sec:Review}
As in~\cite{van2024riemannian,van2025generalized}, we define the indefinite Stiefel manifold as
\begin{equation}
\iSt_{A, J}(p,n) \coloneqq \{X \in \R^{n \times p} \mid X^{\top}AX = J\}
\end{equation}
and regard it as an embedded submanifold of $\R^{n \times p}$.
\footnote{Since we assume that the number of positive/negative eigenvalues of $J$ is not greater than that of $A$ as in Section~\ref{sec:Notation}, $\iSt_{A,J}(p,n) \neq \emptyset$~\cite{van2024riemannian}.}
This is possible by the regular level set theorem~\cite{AbsMahSep2008,boumal2023intromanifolds}.
Specifically, by defining $F \colon \R^{n \times p} \to \Sym(p)$ as $F(X) \coloneqq X^{\top}AX - J$, we have $\iSt_{A,J}(p,n) = F^{-1}(\{0\})$. Furthermore, we have $\D F(X)[Y] = Y^{\top}AX + X^{\top}AY$ for $Y \in \R^{n \times p}$, and $\D F(X)$ at $X \in F^{-1}(\{0\})$ is surjective since for any $S \in \Sym(p)$, $\D F(X)[XJS/2] = (SJ^2 + J^2S)/2 = S$ holds.
Then, the tangent space of $\iSt_{A,J}(p,n)$ at $X \in \iSt_{A,J}(p,n)$ is given in~\cite{van2024riemannian} as
\begin{equation}
T_X \!\iSt_{A,J}(p,n) {}= \ker \D F(X) = \{\xi \in \R^{n \times p} \mid \xi^{\top}AX + X^{\top}A\xi = 0\} 
= \{\Omega AX \mid \Omega \in \Skew(n)\}.
\end{equation}

We would also like to endow $\iSt_{A,J}(p,n)$ with a Riemannian metric to make it a Riemannian manifold.
To this end, we endow an ambient space with a Riemannian metric and then endow $\iSt_{A,J}(p,n)$ with the induced metric.
However, the Riemannian metrics we use in the subsequent discussion, which are the same as those proposed in~\cite{van2025generalized}, are not necessarily defined in whole $\R^{n \times p}$.
In this paper, we therefore consider the open submanifold $\E$ of $\R^{n \times p}$ defined in~\eqref{eq:openE} and regard $\iSt_{A,J}(p,n)$ as an embedded submanifold of $\E$.
See also Remark~\ref{rem:E} below.

\begin{proposition}
The indefinite Stiefel manifold $\iSt_{A,J}(p,n)$ is an embedded submanifold of $\E$, where $\E$ is the open submanifold of $\R^{n \times p}$ defined in~\eqref{eq:openE}.
\end{proposition}
\begin{proof}
We show the inclusion $\iSt_{A,J}(p,n) \subset \E$.
Let $X \in \iSt_{A,J}(p,n)$.
Then, we have $\det(X^{\top}AX) = \det(J) \neq 0$.
It follows that $\ker X \subset \ker (X^{\top}AX) = \{0\}$, implying $\ker X = \{0\}$.
Therefore, $\Rank(X^{\top}X) = \Rank X = p$, meaning $\det(X^{\top}X) \neq 0$.
Thus, $X \in \E$ holds.
Using this inclusion, the fact that $\iSt_{A,J}(p,n)$ is an embedded submanifold of $\E$ is proved in the same way as the proof of $\iSt_{A,J}(p,n)$ being an embedded submanifold of $\R^{n \times p}$ by the regular level set theorem.
\end{proof}
To regard $\iSt_{A,J}(p,n)$ as a Riemannian submanifold of $\E$, we first endow $\E$ with a Riemannian metric and then $\iSt_{A,J}(p,n)$ with the induced metric.
Specifically, we endow $\E$ with the Riemannian metric
\begin{equation}
\label{eq:metric_opensubmanifold}
\langle \xi, \eta\rangle_X \coloneqq \tr(\xi^{\top}G_X \eta), \quad \xi, \eta \in T_X \E = \R^{n \times p}, \ X \in \E,
\end{equation}
where $G \colon \E \ni X \mapsto G_X \in \Sym_{++}(n)$ is smooth.
The induced metric on $\iSt_{A,J}(p,n)$ is then defined as
\begin{equation}
\label{eq:Riemannian_metric}
\langle \xi, \eta\rangle_X \coloneqq \tr(\xi^{\top}G_X\eta), \qquad \xi, \eta \in T_X\!\iSt_{A,J}(p,n),\quad X \in \iSt_{A,J}(p,n).
\end{equation}

{
\begin{remark}
\label{rem:E}
In the literature~\cite{van2024riemannian,van2025generalized}, $\iSt_{A,J}(p,n)$ is regarded as an embedded submanifold of $\R^{n \times p}$ and endowed with the Riemannian metric~\eqref{eq:Riemannian_metric}.
Meanwhile, in this paper, we endow an open submanifold $\E$ of $\R^{n \times p}$ with the Riemannian metric~\eqref{eq:metric_opensubmanifold} and $\iSt_{A,J}(p,n)$ with the Riemannian submanifold structure.
This is because, in this paper, we need to compute the derivative $\D G(X)$ to investigate the second-order geometry of $\iSt_{A,J}(p,n)$.
Indeed, since $\E$ is endowed with the Riemannian metric~\eqref{eq:metric_opensubmanifold}, $G$ is defined in a sufficiently small open ball around $X$ in $\R^{n \times p}$.
Therefore, we have $\D G(X)[Y] = \lim_{t \to 0}(G(X + tY) - G(X))/t$ because $G(X+tY)$ is defined for $t \in \R$ sufficiently close to $0$.
Specific examples of choice of $G_X$ are shown in~\eqref{eq:G1} and~\eqref{eq:G2} below.
\end{remark}
}

For $X \in \iSt_{A,J}(p,n) \subset \E$, the tangent space $T_X\!\iSt_{A,J}(p,n)$ is a linear subspace of the inner product space $T_X \E = \R^{n \times p}$ since $\E$ is an open submanifold of $\R^{n \times p}$~\cite{AbsMahSep2008}.
Therefore, the orthogonal complement of $T_X\!\iSt_{A,J}(p,n)$, which is called the normal space of $\iSt_{A,J}(p,n)$ at $X$, can be defined.
This is written out~\cite{van2024riemannian} as
\begin{align}
N_X\!\iSt_{A,J}(p,n) &\coloneqq \{\zeta \in \R^{n \times p} \mid \langle \zeta, \xi\rangle_X = 0, \ \xi \in T_X \! \iSt_{A,J}(p,n)\} = \{G_X^{-1}AXW \mid W \in \Sym(p)\}.
\end{align}
Since $T_X \E = \R^{n \times p}$ is decomposed as a direct sum $\R^{n \times p} = T_X\!\iSt_{A,J}(p,n) \oplus N_X\!\iSt_{A,J}(p,n)$, we can uniquely decompose any $Y \in \R^{n \times p}$ as $Y = \xi + \zeta$ with $\xi \in T_X \! \iSt_{A,J}(p,n)$ and $\zeta \in N_X \! \iSt_{A,J}(p,n)$.
We can therefore define the orthogonal projection $P_X^{G} \colon \R^{n \times p} \to T_X\!\iSt_{A,J}(p,n)$ onto the tangent space by $P_X^{G}(Y) \coloneqq \xi$.
Specifically, \cite{van2024riemannian} reveals that
\begin{equation}
\label{eq:projection_original}
P_X^{G}(Y) = Y - G_X^{-1}AXU_{X,Y},
\end{equation}
where $U_{X,Y}$ is the solution to the Lyapunov equation~\cite{horn2012matrix}
\begin{equation}
\label{eq:Lyapunov_P}
(X^{\top}AG_X^{-1}AX)U + U(X^{\top}AG_X^{-1}AX) = 2\sym(X^{\top}AY)
\end{equation}
with respect to $U \in \R^{p \times p}$.
Therefore, the Riemannian gradient of a smooth function $f \colon \iSt_{A,J}(p,n) \to \R$ is given by
\begin{equation}
\grad^{G} f(X) = G_X^{-1} \grad_{\mathrm{E}} \bar{f}(X) - G_X^{-1}AX U_f,
\end{equation}
where $U_f$ is the unique solution to the Lyapunov equation
\begin{equation}
\label{eq:Lyapunov_f}
(X^{\top}AG_X^{-1}AX)U + U(X^{\top}AG_X^{-1}AX) = 2\sym(X^{\top}AG_X^{-1}\grad_{\mathrm{E}} \bar{f}(X))
\end{equation}
with respect to $U$.
Here, $\bar{f} \colon \E \to \R$ is a smooth extension of $f$ to $\E$ (i.e., $\bar{f}$ is smooth and its restriction to $\iSt_{A,J}(p,n)$ is equal to $f$) and $\grad_{\mathrm{E}} \bar{f}$ is the Euclidean gradient of $\bar{f}$.
To avoid any confusion regarding notation, note here that, in this paper, we consider the Euclidean gradient $\grad_{\mathrm{E}}$ only for $\bar{f}$ and always consider the Riemannian gradient otherwise.
Although the Lyapunov equations~\eqref{eq:Lyapunov_P} and~\eqref{eq:Lyapunov_f} can be efficiently solved from a numerical perspective especially when $p$ is very small, we note that the orthogonal projection onto a tangent space and gradient of a function are not explicitly written out.
This may be a difficulty in analyzing higher-order derivatives, e.g., the Hessian of a function.

Fortunately, we can avoid solving such Lyapunov equations.
To this end, a recent study~\cite{van2025generalized} proposes two specific choices of $G_X$ as
\begin{equation}
\label{eq:G1}
G_X^{(1)} \coloneqq \frac{1}{\rho}AXX^{\top}A + (A-AXJX^{\top}A)^2
\end{equation}
and
\begin{equation}
\label{eq:G2}
G_X^{(2)} \coloneqq \frac{1}{\rho}AXX^{\top}A + I_n-X(X^{\top}X)^{-1}X^{\top},
\end{equation}
where $\rho > 0$ is a parameter.
A key feature of these two choices is that $(G_X^{(i)})^{-1}AX=\rho XJ$ for $i=1,2$ hold on $\iSt_{A,J}(p,n)$.
Consequently, the Lyapunov equation appearing in the general projection
formula reduces to a simple equation that is easy to solve. This identity is also crucial in
obtaining explicit formulas for the Riemannian gradient and Hessian.
We here prove that the associated $G_X$ indeed defines Riemannian metrics on $\E$.
\begin{proposition}
For $X \in \E$, $G_X^{(1)}$ and $G_X^{(2)}$ in~\eqref{eq:G1} and~\eqref{eq:G2} are symmetric positive-definite matrices.
\end{proposition}
\begin{proof}
Symmetry of the two matrices is trivial.

For arbitrary $v \in \R^n$, we have
\begin{equation}
v^{\top}G_X^{(1)}v = \frac{1}{\rho}\|X^{\top}Av\|_2^2 + \|(A-AXJX^{\top}A)v\|_2^2 \geq 0,
\end{equation}
where $\|\cdot\|_2$ is the $2$-norm.
If $v^{\top}G_X^{(1)}v = 0$, then $X^{\top}Av = (A-AXJX^{\top}A)v = 0$ should hold.
It follows that $Av = AXJX^{\top}Av = 0$, implying $v = 0$ since $A$ is invertible.
Therefore, $G^{(1)}_X$ is positive-definite.
Similarly, we have
\begin{equation}
v^{\top}G_X^{(2)}v = \frac{1}{\rho}\|X^{\top}Av\|_2^2 + \|(I_n - X(X^{\top}X)^{-1}X^{\top})v\|_2^2 \geq 0.
\end{equation}
If $v^{\top}G_X^{(2)}v = 0$, then $X^{\top}Av = (I_n - X(X^{\top}X)^{-1}X^{\top})v = 0$ should hold.
It follows that $v = Xw$ for $w \coloneqq (X^{\top}X)^{-1}X^{\top}v$.
Therefore, $0 = X^{\top}Av = X^{\top}AXw$ holds, implying $w = 0$ since $X^{\top}AX$ is invertible for $X \in \E$.
Thus, we obtain $v = Xw = 0$.
Hence, $G_X^{(2)}$ is positive-definite.
\end{proof}

According to~\cite{van2025generalized}, both of $G_X^{(1)}$ and $G_X^{(2)}$ enable a simpler computation of the corresponding orthogonal projection onto the tangent space $T_X\!\iSt_{A,J}(p,n)$ and Riemannian gradient of a function on $\iSt_{A,J}(p,n)$.
For $X \in \iSt_{A,J}(p,n)$, their inverses can be written as
\begin{equation}
\label{eq:inverse1_0}
{G_X^{(1)}}^{-1} = \rho XX^{\top} + A^{-1}X_{\perp}(X_{\perp}^{\top}X_{\perp})^{-1}X_{\perp}^{\top}A^{-1}
\end{equation}
and
\begin{equation}
\label{eq:inverse2_0}
{G_X^{(2)}}^{-1} = \rho XX^{\top} + A^{-1}X_{\perp}(X_{\perp}^{\top}A^{-1}X_{\perp})^{-1}X_{\perp}^{\top}X_{\perp}(X_{\perp}^{\top}A^{-1}X_{\perp})^{-1}X_{\perp}^{\top}A^{-1},
\end{equation}
where $X_{\perp} \in \R^{n \times (n-p)}$ is an arbitrarily chosen full-rank matrix satisfying $X^{\top}X_{\perp} = 0$.
Then, specifically, the orthogonal projection $P_X^{G^{(i)}}$ for $i = 1, 2$ can be computed as
\begin{equation}
\label{eq:projection_simplified}
P_X^{G^{(i)}}(Y) = Y - XJ\sym(X^{\top}AY).
\end{equation}
Indeed, for $G_X = {G_X^{(i)}}$ with $i = 1, 2$, the Lyapunov equation~\eqref{eq:Lyapunov_P} reduces to
\begin{equation}
\rho U + \rho U = 2\sym(X^{\top}AY)
\end{equation}
since ${G_X^{(i)}}^{-1}AX = \rho XX^{\top}AX = \rho XJ$ and $X^{\top}A{G_X^{(i)}}^{-1}AX = \rho X^{\top}AXJ = \rho J^2 = \rho I_p$ holds from~\eqref{eq:inverse1_0} and~\eqref{eq:inverse2_0}.
It follows that $U = \rho^{-1}\sym(X^{\top}AY)$, together with \eqref{eq:projection_original} leads to
\begin{equation}
P_X^{G^{(i)}}(Y) = Y - {G_X^{(i)}}^{-1}AXU = Y - \rho XJ \rho^{-1}\sym(X^{\top}AY) = Y - XJ \sym(X^{\top}AY).
\end{equation}

Therefore, the corresponding Riemannian gradient of a smooth function $f \colon \iSt_{A,J}(p,n) \to \R$, which is given by $\grad^{G^{(i)}}f(X) = P^{G^{(i)}}_X({G^{(i)}_X}^{-1}\grad_{\mathrm{E}} \bar{f}(X))$, is written as follows.

\begin{proposition}
Let $f\colon\iSt_{A,J}(p,n)\to\R$ be a smooth function and $\bar f\colon\E\to\R$ be a smooth extension of $f$.
The Riemannian gradient of $f$ with respect to the Riemannian metric~\eqref{eq:Riemannian_metric}, where $G_X = G_X^{(i)}$, $i = 1, 2$, is written as
\begin{equation}
\grad^{G^{(i)}}f(X) = {G^{(i)}_X}^{-1}\grad_{\mathrm{E}} \bar{f}(X) - XJ\sym(X^{\top}A{G^{(i)}_X}^{-1}\grad_{\mathrm{E}} \bar{f}(X)).\label{eq:Rgrad}
\end{equation}
\end{proposition}

Here, we can simplify~\eqref{eq:inverse1_0} and~\eqref{eq:inverse2_0} as follows.

\begin{proposition}
\label{prop:inverse}
Let $X \in \iSt_{A,J}(p,n)$.
For $G_X^{(1)}$ and $G_X^{(2)}$ defined by~\eqref{eq:G1} and \eqref{eq:G2}, we have
\begin{equation}
\label{eq:inverse1}
{G_X^{(1)}}^{-1} = \rho XX^{\top} + A^{-1}(I_n - X(X^{\top}X)^{-1}X^{\top})A^{-1}
\end{equation}
and
\begin{equation}
\label{eq:inverse2}
{G_X^{(2)}}^{-1} = \rho XX^{\top} + (I_n - XJX^{\top}A)(I_n - XJX^{\top}A)^{\top}.    
\end{equation}
\end{proposition}

\begin{proof}
Since $X \in \R^{n \times p}$ and $X_{\perp} \in \R^{n \times (n-p)}$ are both of full rank, there exist $S \in \R^{p \times n}$ and $T \in \R^{(n-p) \times n}$ such that $\begin{bmatrix}
    X & X_{\perp}
\end{bmatrix}\begin{bmatrix}
    S\\T
\end{bmatrix} = I_n$, i.e.,
$XS + X_{\perp}T = I_n$.
Here, multiplying the equality by $X^{\top}$ from the left gives $S = (X^{\top}X)^{-1}X^{\top}$.
Similarly, we have $T = (X_{\perp}^{\top}X_{\perp})^{-1}X_{\perp}^{\top}$.
Therefore, it holds that
\begin{equation}
X(X^{\top}X)^{-1}X^{\top} + X_{\perp}(X_{\perp}^{\top}X_{\perp})^{-1}
X_{\perp}^{\top} = I_n.
\end{equation}
Using this identity, \eqref{eq:inverse1_0} reduces to~\eqref{eq:inverse1}.

Furthermore, the following formula is given in~\cite{van2024riemannian}:
\begin{equation}
\label{eq:identity}
XJX^{\top}A + A^{-1}X_{\perp}(X_{\perp}^{\top}A^{-1}X_{\perp})^{-1}X_{\perp}^{\top} = I_n.
\end{equation}
The second equality~\eqref{eq:inverse2} is a direct consequence of~\eqref{eq:inverse2_0} and~\eqref{eq:identity}.
\end{proof}

On the standard Stiefel manifold $\St(p,n) = \iSt_{I_n, I_p}(p,n)$, the matrix $XX^{\top} \colon \R^n \to \R^n$ for $X \in \St(p,n)$ (i.e., $X^{\top}X = I_p$) is the orthogonal projection onto $\Span(X)$ since for any $Xc \in \Span(X)$ and $X_{\perp}d \in \Span(X)^{\perp}$ with $c \in \R^p$ and $d \in \R^{n-p}$, it holds that $(XX^{\top})(Xc) = X(X^{\top}X)c = Xc$ and $(XX^{\top})(X_{\perp}d) = X(X^{\top}X_{\perp})d = 0$.
However, for $\iSt_{A,J}(p,n)$, the matrix $XX^{\top}$ in~\eqref{eq:inverse1} and~\eqref{eq:inverse2} is not an orthogonal projection in general. Meanwhile, the second terms in both equations contain orthogonal projections (see also Remark~\ref{rem:projection} below).
To obtain a clearer perspective, we define some matrices depending on $X \in \iSt_{A,J}(p,n)$ and rewrite $G_X^{(i)}$ and the inverse ${G_X^{(i)}}^{-1}$ for $i = 1, 2$.

\begin{proposition}
\label{prop:G_summary}
For $X \in \iSt_{A,J}(p,n)$, we define symmetric matrices
\begin{equation}
M_X \coloneqq (X^{\top}X)^{-1}, \quad \Pi_X \coloneqq I_n - XM_X X^{\top}, \quad S_X \coloneqq A - AXJX^{\top}A
\end{equation}
and square matrix
\begin{equation}
Q_X \coloneqq A^{-1}S_X = I_n - XJX^{\top}A.
\end{equation}
It holds that
\begin{equation}
\label{eq:PiXSX}
\Pi_X X = S_X X = 0, \quad \Pi_X^2 = \Pi_X, \quad Q_X^2 = Q_X.
\end{equation}
Then, \eqref{eq:G1} and \eqref{eq:G2} can be rewritten as
\begin{equation}
\label{eq:G1_2}
G^{(1)}_X = \frac{1}{\rho} AXX^{\top}A + S_X^2
\end{equation}
and
\begin{equation}
\label{eq:G2_2}
G^{(2)}_X = \frac{1}{\rho} AXX^{\top}A + \Pi_X.
\end{equation}
Furthermore, their inverses are written as
\begin{equation}
\label{eq:inverse1_2}
{G^{(1)}_X}^{-1} = \rho XX^{\top} + A^{-1}\Pi_X A^{-1}
\end{equation}
and \noeqref{eq:G2_2}\noeqref{eq:inverse1_2}
\begin{equation}
\label{eq:inverse2_2}
{G^{(2)}_X}^{-1} = \rho XX^{\top} + Q_X Q_X^{\top} = \rho XX^{\top} + A^{-1}S_X^2 A^{-1}.
\end{equation}
Both inverses ${G^{(i)}_X}^{-1}$ with $i = 1, 2$ satisfy
\begin{equation}
X^{\top}A{G^{(i)}_X}^{-1}K = \rho J X^{\top}K
\end{equation}
for any $K \in \R^{n \times p}$.
\end{proposition}

\begin{proof}
We have $\Pi_X X = X - X(X^{\top}X)^{-1}X^{\top}X = 0$, $S_X X = AX - AXJX^{\top}AX = AX - AXJ^2 = 0$, $\Pi_X^2 = I_n - 2XM_X X^{\top} + XM_XX^{\top}XM_XX^{\top} = I_n - XM_XX^{\top} = \Pi_X$, and $Q_X^2 = I_n - 2XJX^{\top}A + XJX^{\top}AXJX^{\top}A = I_n - XJX^{\top}A = Q_X$ from $X^{\top}AX = J$ and $J^2 = I_p$, completing the proof of \eqref{eq:PiXSX}.
Equations~\eqref{eq:G1_2}--\eqref{eq:inverse2_2} are straightforward from \eqref{eq:G1}, \eqref{eq:G2}, \eqref{eq:inverse1}, and \eqref{eq:inverse2}.

Finally, for $K \in \R^{n \times p}$,
\begin{equation}
X^{\top}A{G^{(1)}_X}^{-1}K = X^{\top}A(\rho XX^{\top} + A^{-1}\Pi_X A^{-1})K = \rho JX^{\top}K
\end{equation}
and
\begin{equation}
X^{\top}A{G^{(2)}_X}^{-1}K = X^{\top}A(\rho XX^{\top} + Q_XQ_X^{\top})K = \rho JX^{\top}K,
\end{equation}
completing the proof.
\end{proof}

{
\begin{remark}
\label{rem:projection}
Note that $\Pi_X$ in~\eqref{eq:inverse1_2} and $Q_X$ in~\eqref{eq:inverse2_2} are idempotent from~\eqref{eq:PiXSX}.
Therefore, they are projection matrices.
Specifically, as for the former, $\Pi_X = I_n - X(X^{\top}X)^{-1}X^{\top}$ is the orthogonal projection onto $\Span(X)^{\perp}$ with respect to the standard inner product in $\R^n$ since $\Pi_X X = 0$ and $\Pi_X X_{\perp} = X_{\perp}$ imply $\ker(\Pi_X) = \Span(X)$ and $\Span(\Pi_X) = \Span(X_{\perp}) = \Span(X)^{\perp}$.
As for the latter, $Q_X$ is the orthogonal projection with respect to the (possibly indefinite) inner product in $\R^n$ defined via $A$ since $Q_X X = A^{-1}S_X X = 0$ and $Q_X^{\top}AX = S_X X = 0$ imply $\ker(Q_X) = \Span(X)$ and $\Span(Q_X) = \Span(X)^{\perp_A}$, where
\begin{align}
\Span(X)^{\perp_A} &\coloneqq \{v \in \R^n \mid \forall w \in \Span(X), v^{\top}Aw = 0\} \\
&= \{v \in \R^n \mid \forall u \in \R^p, v^{\top}AXu = 0\} \\
&= \{v \in \R^n \mid X^{\top}Av = 0\}.
\end{align}
Indeed,
$Q_X X = 0$ yields $\ker(Q_X) \supset \Span(X)$, and $\ker(Q_X) \subset \Span(X)$ also holds since $v \in \ker(Q_X)$ satisfies $v = XJX^{\top}Av \in \Span(X)$.
Furthermore, $Q_X^{\top}AX = 0$ yields $\Span(Q_X) \subset \Span(X)^{\perp_A}$, and $\Span(Q_X) \supset \Span(X)^{\perp_A}$ also holds since $v \in \Span(X)^{\perp_A}$ satisfies $v = v - XJX^{\top}Av = Q_X v \in \Span(Q_X)$.
Although $Q_X=A^{-1}S_X$ is idempotent, the symmetric matrix $S_X$ is not idempotent in general and, hence, is not generally a projection matrix.
\end{remark}
}

\section{Second-order geometry of the indefinite Stiefel manifold}
\label{sec:second-order}
One of our goals is to develop Newton-type methods for optimization problems on the indefinite Stiefel manifold $\iSt_{A,J}(p,n)$.
To this end, it is necessary to be able to compute the Riemannian Hessian of a given objective function defined on $\iSt_{A,J}(p,n)$.
In this section, we compute the Levi-Civita connection on $\iSt_{A,J}(p,n)$ before computing the Hessian of a function on $\iSt_{A,J}(p,n)$.

As introduced in the previous section, we endow $\iSt_{A,J}(p,n)$ with a Riemannian metric~\eqref{eq:Riemannian_metric} with $G_X$ being $G_X^{(1)}$ in~\eqref{eq:G1} or $G_X^{(2)}$ in~\eqref{eq:G2}.
We derive specific formulas for the Levi-Civita connection and Hessian with respect to $G_X^{(1)}$ and $G_X^{(2)}$.

\subsection{Levi-Civita connection on $\E$}

Note that the indefinite Stiefel manifold is an embedded submanifold of $\E$.
We derive a formula for the Levi-Civita connection $\bar{\nabla} \colon \mathfrak{X}(\E) \times \mathfrak{X}(\E) \to \mathfrak{X}(\E)$ on $\E$ with respect to the Riemannian metric~\eqref{eq:metric_opensubmanifold} in this subsection and then that on $\iSt_{A,J}(p,n)$ with respect to the induced metric~\eqref{eq:Riemannian_metric} in the subsequent subsections.

We define a map $G \colon \E \to \Sym_{++}(n)$ by $G(X) \coloneqq G_X$.
Throughout this subsection, let $X \in \E$.
For vector fields $\bar{U}, \bar{V}, \bar{W} \in \mathfrak{X}(\E)$, Koszul's formula~\cite{lee2018introduction} yields
\begin{align}
2\langle \bar{\nabla}_{\bar{U}} \bar{V}, \bar{W}\rangle_X ={} &\bar{U}\langle \bar{V}, \bar{W}\rangle_X + \bar{V}\langle \bar{W}, \bar{U}\rangle_X - \bar{W}\langle \bar{U}, \bar{V}\rangle_X\\
&+ \langle[\bar{U}, \bar{V}], \bar{W}\rangle_X - \langle [\bar{V}, \bar{W}], \bar{U}\rangle_X + \langle [\bar{W}, \bar{U}], \bar{V}\rangle_X.\label{eq:Koszul}
\end{align}
Here, $\bar{U}\langle \bar{V}, \bar{W}\rangle_X$ means a real number obtained by applying the vector field $\bar{U}$ to a function $X \mapsto \langle \bar{V}, \bar{W}\rangle_X = \tr(\bar{V}(X)^{\top}G_X \bar{W}(X))$ on $\E$, i.e.,
\begin{align}
\bar{U}\langle \bar{V}, \bar{W}\rangle_X
= \tr(&\D \bar{V}(X)[\bar{U}(X)]^{\top}G_X \bar{W}(X) \\
&+ \bar{V}(X)^{\top}\D G(X)[\bar{U}(X)]\bar{W}(X) + \bar{V}(X)^{\top}G_X \D \bar{W}(X)[\bar{U}(X)]).
\end{align}
The Lie bracket $[\bar{U}, \bar{V}]$ is a vector field on $\E$ defined to satisfy
\begin{equation}
[\bar{U}, \bar{V}](\bar{f}) \coloneqq \bar{U}(\bar{V}\bar{f}) - \bar{V}(\bar{U}\bar{f})
\end{equation}
for any smooth function $\bar{f} \colon \E \to \R$.
Regarding the Levi-Civita connection $\bar{\nabla}$, for $\xi \coloneqq \bar{U}(X), \eta \coloneqq \bar{V}(X) \in T_X \E = \R^{n \times p}$, we define
\begin{equation}
\Gamma_X(\xi, \eta) \coloneqq (\bar{\nabla}_{\bar{U}} \bar{V})(X) - \D \bar{V}(X)[\xi],
\end{equation}
where
\begin{equation}
\D \bar{V}(X)[\xi] \coloneqq \lim_{t \to 0}\frac{\bar{V}(X + t\xi) - \bar{V}(X)}{t}.
\end{equation}
This directional derivative makes sense since $X + t\xi$ belongs to $\E$ for $t$ sufficiently close to $0$ and therefore $\bar{V}(X+t\xi)$ is defined.
Note that $\Gamma_X(\xi, \eta)$ does not depend on $\bar{U}$ or $\bar{V}$ except the values $\xi = \bar{U}(X)$ and $\eta = \bar{V}(X)$ at $X$.
The term $\Gamma_X(\xi, \eta)$ in $(\bar{\nabla}_{\bar{U}}\bar{V})(X) = \D \bar{V}(X)[\xi] + \Gamma_X(\xi, \eta)$ is expressed with Christoffel symbols when using local coordinates.\footnote{The matrix entries provide global coordinates on the open submanifold $\E$.}
Therefore, we call $\Gamma_X \colon \R^{n \times p} \times \R^{n \times p} \to \R^{n \times p}$ the Christoffel function as in~\cite{nguyen2024second}.
Note that $\Gamma_X$ is bilinear.
To obtain an explicit formula for the Levi-Civita connection $\bar{\nabla}_{\bar{U}}\bar{V}$, it suffices to investigate $\Gamma_X$.

\begin{proposition}
The Levi-Civita connection $\bar{\nabla}$ on the Riemannian manifold $\E$ endowed with the Riemannian metric \eqref{eq:metric_opensubmanifold} acts on vector fields $\bar{U}, \bar{V} \in \mathfrak{X}(\E)$ with $\bar{U}(X) = \xi, \bar{V}(X) = \eta \in T_X \E = \R^{n \times p}$ as
\begin{align}
(\bar{\nabla}_{\bar{U}} \bar{V})(X) &= \D \bar{V}(X)[\xi] + \Gamma_X(\xi, \eta)\\
&= \D \bar{V}(X)[\xi] + \frac{1}{2}(G_X^{-1}(\D G(X)[\xi] \eta + \D G(X)[\eta]\xi) - \D G(X)^*[\sym(\xi \eta^{\top})]),
\end{align}
where $\D G(X)^* \colon \Sym(n) \to \R^{n \times p}$ is the adjoint of $\D G(X) \colon \R^{n \times p} \to \Sym(n)$ with respect to the inner product $\langle \cdot, \cdot\rangle_X$ in $\R^{n \times p}$ and the standard inner product in $\Sym(n)$ satisfying $\langle\D G(X)^*[S], \zeta\rangle_X = \tr(S\D G(X)[\zeta])$ for any $\zeta \in \R^{n \times p}$ and $S \in \Sym(n)$.
\end{proposition}

\begin{proof}
To compute $\Gamma_X(\xi, \eta)$, we specifically choose constant vector fields $\tilde{U} \equiv \xi$ and $\tilde{V} \equiv \eta$.
Then, we have
\begin{equation}
(\bar{\nabla}_{\tilde{U}}\tilde{V})(X) = \D \tilde{V}(X)[\xi] + \Gamma_X(\xi, \eta) = \Gamma_X(\xi, \eta),
\end{equation}
and $[\bar{W}, \tilde{U}](X) = -\D \bar{W}(X)[\xi]$ yields
\begin{align}
&\tilde{U}\langle\tilde{V}, \bar{W}\rangle_X + \langle[\bar{W}, \tilde{U}], \tilde{V}\rangle_X\\
={}& \tr(\eta^{\top}\D G(X)[\xi]\bar{W}(X) + \eta^{\top}G_X \D \bar{W}(X)[\xi]) - \tr(\D \bar{W}(X)[\xi]^{\top}G_X\eta)\\
={}& \tr(\eta^{\top}\D G(X)[\xi]\bar{W}(X)).
\end{align}
Similarly, $[\tilde{V}, \bar{W}](X) = \D \bar{W}(X)[\eta]$ yields
\begin{align}
&\tilde{V}\langle \bar{W}, \tilde{U}\rangle_X - \langle[\tilde{V}, \bar{W}], \tilde{U}\rangle_X\\
={}& \tr(\D \bar{W}(X)[\eta]^{\top}G_X\xi + \bar{W}(X)^{\top}\D G(X)[\eta]\xi) - \tr(\D \bar{W}(X)[\eta]^{\top}G_X\xi)\\
={}& \tr(\bar{W}(X)^{\top}\D G(X)[\eta]\xi).
\end{align}
Furthermore, we have $[\tilde{U}, \tilde{V}] = 0$.
Taking all of them into account, it follows from Koszul's formula~\eqref{eq:Koszul} that
\begin{equation}
2\langle \Gamma_X(\xi, \eta), \bar{W}\rangle_X = \tr(\eta^{\top}\D G(X)[\xi]\bar{W}(X)) + \tr(\bar{W}(X)^{\top}\D G(X)[\eta]\xi) -\tr(\xi^{\top}\D G(X)[\bar{W}(X)]\eta).
\end{equation}
This identity implies that the right-hand side, and hence the left-hand side, depend on $\bar{W}$ only at the value $\bar{W}(X)$ at $X$.
Therefore, for arbitrary $\zeta \in \R^{n \times p}$, we set $\bar{W}$ as a constant vector field $\tilde{W} \equiv \zeta$ to obtain
\begin{align}
2\langle \Gamma_X(\xi, \eta), \zeta\rangle_X &= \tr(\eta^{\top}\D G(X)[\xi]\zeta) + \tr(\zeta^{\top}\D G(X)[\eta]\xi) - \tr(\xi^{\top}\D G(X)[\zeta]\eta)\\
&= \tr(\eta^{\top}\D G(X)[\xi]G_X^{-1}G_X\zeta + \zeta^{\top}G_XG_X^{-1}\D G(X)[\eta]\xi) - \tr(\eta\xi^{\top} \D G(X)[\zeta])\\
&= \langle G_X^{-1}\D G(X)[\xi]\eta + G_X^{-1} \D G(X)[\eta]\xi, \zeta\rangle_X - \tr(\sym(\xi\eta^{\top}) \D G(X)[\zeta])\\
&= \langle G_X^{-1}(\D G(X)[\xi] \eta + \D G(X)[\eta]\xi) - \D G(X)^* [\sym(\xi\eta^{\top})], \zeta\rangle_X. \label{eq:IdentifyingGamma}
\end{align}
Since \eqref{eq:IdentifyingGamma} holds for any $\zeta \in \R^{n \times p}$, the Christoffel function $\Gamma_X$ is written out as
\begin{equation}
\Gamma_X(\xi, \eta) = \frac{1}{2}(G_X^{-1}(\D G(X)[\xi]\eta + \D G(X)[\eta]\xi) - \D G(X)^*[\sym(\xi\eta^{\top})]),
\end{equation}
completing the proof.
\end{proof}

\subsection{The derivatives of metric matrices and their adjoint}
To further clarify the expression of the Levi-Civita connection, we compute the derivative $\D G(X)$ and its adjoint $\D G(X)^*$ for both cases $G_X = G_X^{(1)}, G_X^{(2)}$.
As in Proposition~\ref{prop:G_summary}, we define $M_X \coloneqq (X^{\top}X)^{-1}$, $\Pi_X = I_n - XM_X X^{\top}$, and $S_X \coloneqq A - AXJX^{\top}A$.
Let $G^{(i)} \colon \E \to \Sym_{++}(n)$ be a map defined by $G^{(i)}(X) \coloneqq G_X^{(i)}$ for $i = 1, 2$.
Then, their derivatives can be computed as
\begin{align}
\D G^{(1)}_X[\zeta] = &\frac{1}{\rho}A(\zeta X^{\top} + X\zeta^{\top})A -A(\zeta JX^{\top} + XJ\zeta^{\top})AS_X - S_X A(\zeta JX^{\top} + XJ\zeta^{\top})A\\
=& \frac{2}{\rho}A\sym(X\zeta^{\top})A - 4\sym(A\sym(XJ\zeta^{\top})AS_X)
\end{align}
and
\begin{align}
\D G^{(2)}_X[\zeta] = &\frac{1}{\rho}A(\zeta X^{\top} + X\zeta^{\top})A -\zeta M_XX^{\top} + XM_X(\zeta^{\top}X + X^{\top}\zeta)M_XX^{\top} -XM_X\zeta^{\top}\\
=& \frac{2}{\rho}A\sym(X\zeta^{\top})A - 2\sym(XM_X\zeta^{\top}) + 2XM_X\sym(\zeta^{\top}X)M_XX^{\top}.
\end{align}

To compute their adjoint, we take the inner product of $\D G^{(i)}(X)[\zeta]$ and an arbitrary $H \in \Sym(n)$, and compute it as follows, noting the relationship $\tr(\sym(B)H) = \tr(BH)$ for arbitrary $B \in \R^{n \times n}$:
For the inner product~\eqref{eq:metric_opensubmanifold} with $G \coloneqq G^{(1)}$, we have
\begin{align}
& \tr(\D G^{(1)}(X)[\zeta]H) \\
={}& \frac{2}{\rho}\tr(X\zeta^{\top}AHA) - 4\tr(A\sym(XJ\zeta^{\top})AS_X H)\\
={}& \frac{2}{\rho}\tr(X\zeta^{\top}AHA) - 4\tr(\sym(XJ\zeta^{\top})\sym(AS_X HA))\\
={}& \frac{2}{\rho}\tr(X\zeta^{\top}AHA) - 4\tr(XJ\zeta^{\top}\sym(AS_X HA))\\
={}& \frac{2}{\rho}\tr(\zeta^{\top}G_X^{(1)}{G_X^{(1)}}^{-1}AHAX) - 4\tr(\zeta^{\top}G_X^{(1)}{G_X^{(1)}}^{-1}\sym(AS_X HA)XJ)\\
={}& \bigg\langle\zeta, {G_X^{(1)}}^{-1}\bigg(\frac{2}{\rho}AHAX - 4\sym(AS_X HA)XJ\bigg)\bigg\rangle_X
\end{align}
and for~\eqref{eq:metric_opensubmanifold} with $G \coloneqq G^{(2)}$, we have
\begin{align}
&\tr(\D G^{(2)}(X)[\zeta]H) \\
={}& \frac{2}{\rho}\tr(X\zeta^{\top}AHA) - 2\tr(XM_X \zeta^{\top}H)+ 2\tr(XM_X \sym(\zeta^{\top}X)M_X X^{\top}H)\\
={}& \frac{2}{\rho}\tr(X\zeta^{\top}AHA) - 2\tr(XM_X \zeta^{\top}H)+ 2\tr(\zeta^{\top}XM_X X^{\top}HXM_X )\\
={}& \frac{2}{\rho}\tr(\zeta^{\top}G_X^{(2)}{G_X^{(2)}}^{-1}AHAX) \\
& - 2\tr(\zeta^{\top}G_X^{(2)}{G_X^{(2)}}^{-1}HXM_X ) + 2\tr(\zeta^{\top}G_X^{(2)}{G_X^{(2)}}^{-1}XM_X X^{\top}HXM_X )\\
={}& \bigg\langle \zeta, {G_X^{(2)}}^{-1}\bigg(\frac{2}{\rho}AHAX - 2HXM_X  + 2XM_X X^{\top}HXM_X \bigg)\bigg\rangle_X.
\end{align}
Therefore, we obtain
\begin{equation}
\D G^{(1)}(X)^*[H] = {G_X^{(1)}}^{-1}\bigg(\frac{2}{\rho}AHAX - 4\sym(AS_X HA)XJ\bigg)
\end{equation}
and
\begin{equation}
\D G^{(2)}(X)^*[H] = {G_X^{(2)}}^{-1}\bigg(\frac{2}{\rho}AHAX - 2HXM_X  + 2XM_X X^{\top}HXM_X \bigg).
\end{equation}

\subsection{Levi-Civita connection on the indefinite Stiefel manifold}
Before proceeding to computing the Levi-Civita connection on the indefinite Stiefel manifold $\iSt_{A,J}(p,n)$, we further elaborate on $\Gamma_X$.
In what follows, we assume that $X \in \iSt_{A,J}(p,n)$, from which we can use some additional properties such as $X^{\top}AX = J$ and $\Pi_X X = S_X X = 0$.

Since $\Gamma_X$ depends on the choice of $G$, we denote $\Gamma_X$ by $\Gamma_X^{(i)}$ when $G \coloneqq G^{(i)}$ for $i = 1, 2$.
Then, we have
\begin{align}
\Gamma^{(1)}_X(\xi, \eta) ={}& {G^{(1)}_X}^{-1}(\rho^{-1}A\sym(X\xi^{\top})A-2\sym(A\sym(XJ\xi^{\top})AS_X))\eta\\
&+ {G^{(1)}_X}^{-1}(\rho^{-1}A\sym(X\eta^{\top})A-2\sym(A\sym(XJ\eta^{\top})AS_X))\xi\\
&- {G^{(1)}_X}^{-1}(\rho^{-1}A\sym(\xi\eta^{\top})AX - 2\sym(AS_X\sym(\xi\eta^{\top})A)XJ)\\
={}& {G^{(1)}_X}^{-1}(\rho^{-1}AB'_{X,\xi,\eta}
-2C_{X,\xi,\eta}),
\end{align}
where we have defined $B'_{X, \xi, \eta}, C_{X,\xi,\eta} \in \R^{n \times p}$ as
\begin{equation}
\label{eq:B0}
B'_{X,\xi,\eta} \coloneqq \sym(X\xi^{\top})A\eta + \sym(X\eta^{\top})A\xi - \sym(\xi\eta^{\top})AX
\end{equation}
and
\begin{align}
C_{X, \xi, \eta} \coloneqq{}& \sym(A\sym(XJ\xi^{\top})AS_X)\eta + \sym(A\sym(XJ\eta^{\top})AS_X)\xi - \sym(AS_X\sym(\xi\eta^{\top})A)XJ.
\label{eq:C}
\end{align}
Similarly, regarding $\Gamma^{(2)}_X$,
\begin{align}
\Gamma^{(2)}_X(\xi,\eta) ={}& {G_X^{(2)}}^{-1}(\rho^{-1}A\sym(X\xi^{\top})A - \sym(XM_X\xi^{\top})  + XM_X\sym(\xi^{\top}X)M_X X^{\top})\eta\\
&+ {G^{(2)}_X}^{-1}(\rho^{-1}A\sym(X\eta^{\top})A - \sym(XM_X\eta^{\top})  + XM_X\sym(\eta^{\top}X)M_X X^{\top})\xi\\
&- {G^{(2)}_X}^{-1}(\rho^{-1}A\sym(\xi\eta^{\top})AX-\sym(\xi\eta^{\top})XM_X + XM_X X^{\top}\sym(\xi\eta^{\top})XM_X)\\
={}& {G^{(2)}_X}^{-1}(\rho^{-1}AB'_{X,\xi,\eta} - D_{X,\xi,\eta} + XM_X E_{X, \xi,\eta}),
\end{align}
where we have defined $D_{X, \xi, \eta} \in \R^{n \times p}$ and $E_{X,\xi,\eta} \in \R^{p \times p}$ as
\begin{equation}
\label{eq:D}
D_{X,\xi,\eta} \coloneqq \sym(XM_X\xi^{\top})\eta + \sym(XM_X\eta^{\top})\xi - \sym(\xi\eta^{\top})XM_X\end{equation}
and
\begin{equation}
E_{X,\xi,\eta} \coloneqq \sym(\xi^{\top}X)M_X X^{\top}\eta + \sym(\eta^{\top}X)M_X X^{\top}\xi - X^{\top}\sym(\xi\eta^{\top})XM_X.   
\end{equation}

Furthermore, using Proposition~\ref{prop:G_summary}, it follows from~\eqref{eq:inverse1_2} and~\eqref{eq:inverse2_2} that
\begin{align}
\Gamma^{(1)}_X(\xi, \eta) &= (\rho XX^{\top} + A^{-1}\Pi_X A^{-1})(\rho^{-1}AB'_{X,\xi,\eta} - 2C_{X,\xi,\eta})\\
&= \frac{1}{\rho}A^{-1}\Pi_XB'_{X,\xi,\eta} -2 \rho XX^{\top}C_{X,\xi,\eta}+XX^{\top}AB'_{X,\xi,\eta} -2 A^{-1}\Pi_X A^{-1}C_{X,\xi,\eta} \label{eq:Gamma1}
\end{align}
and
\begin{align}
\Gamma^{(2)}_X(\xi,\eta)
={}& (\rho XX^{\top} + A^{-1}S_X^2A^{-1})(\rho^{-1}AB'_{X,\xi,\eta} - D_{X,\xi,\eta} + XM_X E_{X,\xi,\eta})\\
={}& \frac{1}{\rho}A^{-1}S_X^2B'_{X,\xi,\eta} + \rho X(-X^{\top}D_{X,\xi,\eta} + E_{X,\xi,\eta})\\
&+ XX^{\top}AB'_{X,\xi,\eta} +A^{-1}S_X(I_n - AXJX^{\top})(-D_{X,\xi,\eta} + XM_XE_{X,\xi,\eta}), \label{eq:Gamma2}
\end{align}
respectively.
To simplify them, we note the following lemma.
\begin{lemma}
For $X \in \iSt_{A,J}(p,n)$ and $\xi, \eta \in T_X\!\iSt_{A,J}(p,n)$, let
\begin{equation}
B_{X,\xi,\eta} \coloneqq \xi X^{\top}A\eta + \eta X^{\top}A\xi.
\end{equation}
Then, it holds that
\begin{equation}
B'_{X,\xi,\eta} = X\sym(\xi^{\top}A\eta) + B_{X,\xi,\eta}, \quad \Pi_X B'_{X,\xi,\eta} = \Pi_X B_{X,\xi,\eta}, \quad S_X^2 B'_{X,\xi,\eta} = S_X^2 B_{X,\xi,\eta}.
\end{equation}
Furthermore,
\begin{equation}
X^{\top}AB'_{X,\xi,\eta} = J\sym(\xi^{\top}A\eta) + 2\sym(X^{\top}A\xi X^{\top}A\eta)
\end{equation}
holds.
Regarding $D_{X,\xi, \eta}$ and $E_{X,\xi,\eta}$,
\begin{equation}
\label{eq:D_2}
D_{X,\xi,\eta} = XM_X\sym(\xi^{\top}\eta) + \xi\skew(M_X X^{\top}\eta) + \eta\skew(M_X X^{\top}\xi),
\end{equation}
\begin{equation}
-X^{\top}D_{X,\xi,\eta} + E_{X,\xi,\eta} = -\sym(\xi^{\top}\Pi_X \eta),
\end{equation}
\begin{equation}
A^{-1}S_X (-D_{X,\xi,\eta} + XM_X E_{X,\xi,\eta}) = -A^{-1}S_X(\xi\skew(M_XX^{\top}\eta) + \eta\skew(M_XX^{\top}\xi))
\end{equation}
hold.
\end{lemma}

\begin{proof}
We have
\begin{align}
B'_{X,\xi,\eta} &= \frac{1}{2}(X\xi^{\top}A\eta + \xi X^{\top}A\eta + X\eta^{\top}A\xi + \eta X^{\top}A\xi - \xi\eta^{\top}AX - \eta\xi^{\top}AX)\\
&= \frac{1}{2}(X(\xi^{\top}A\eta + \eta^{\top}A\xi) + \xi(X^{\top}A\eta - \eta^{\top}AX) + \eta(X^{\top}A\xi - \xi^{\top}AX))\\
&= X\sym(\xi^{\top}A\eta) + \xi \skew(X^{\top}A\eta) + \eta \skew(X^{\top}A\xi)\\
&= X \sym(\xi^{\top}A\eta) + B_{X,\xi,\eta},
\end{align}
where we note that $\skew(X^{\top}A\eta) = X^{\top}A\eta$ and $\skew(X^{\top}A\xi) = X^{\top}A\xi$ since $\eta$ and $\xi$ are tangent at $X$.
It follows from $\Pi_XX = 0$ and $S_X X = 0$ that
$\Pi_X B'_{X,\xi,\eta} = \Pi_X B_{X,\xi,\eta}$
and
$S_X^2 B'_{X,\xi,\eta} = S_X^2 B_{X,\xi,\eta}$.
Furthermore, we have
\begin{align}
X^{\top}AB'_{X,\xi,\eta} &= X^{\top}AX\sym(\xi^{\top}A\eta) + X^{\top}A\xi X^{\top}A\eta + X^{\top}A\eta X^{\top}A\xi\\
&= J \sym(\xi^{\top}A\eta) + (X^{\top}A\xi)(X^{\top}A\eta) + (-X^{\top}A\eta)^{\top}(-X^{\top}A\xi)^{\top}\\
&= J\sym(\xi^{\top}A\eta) + 2\sym(X^{\top}A\xi X^{\top}A\eta).
\end{align}
Regarding $D_{X,\xi,\eta}$,
\begin{align}
D_{X,\xi,\eta} &=
\frac{1}{2}(XM_X\xi^{\top}\eta + \xi M_X X^{\top}\eta + XM_X \eta^{\top}\xi + \eta M_X X^{\top}\xi - \xi \eta^{\top}XM_X - \eta\xi^{\top}XM_X)\\
&= XM_X \sym(\xi^{\top}\eta) + \xi\skew(M_X X^{\top}\eta) + \eta\skew(M_X X^{\top}\xi)
\end{align}
holds.
Using $X^{\top}XM_X = I_p$, we have
\begin{align}
&-X^{\top}D_{X,\xi,\eta} + E_{X,\xi,\eta}\\
={}& -\sym(\xi^{\top}\eta) - \frac{1}{2} (X^{\top}\xi M_X X^{\top}\eta + X^{\top}\eta M_X X^{\top}\xi) + X^{\top}\sym(\xi\eta^{\top})XM_X\\
& + \frac{1}{2}(\xi^{\top}XM_X X^{\top}\eta + X^{\top}\xi M_X X^{\top}\eta + \eta^{\top}XM_X X^{\top}\xi + X^{\top}\eta M_XX^{\top}\xi) - X^{\top}\sym(\xi\eta^{\top})XM_X\\
={}& -\sym(\xi^{\top}\eta) + \sym(\xi^{\top}XM_X X^{\top}\eta)\\
={}& -\sym(\xi^{\top}\Pi_X \eta).
\end{align}
Finally, it follows from $-D_{X,\xi,\eta} + XM_X E_{X,\xi,\eta} = (\Pi_X + XM_XX^{\top})(-D_{X,\xi,\eta} + XM_X E_{X,\xi,\eta})$, $S_X X = 0$, and $S_X\Pi_X = S_X(I_n - XM_XX^{\top}) = S_X$ that
\begin{align}
A^{-1}S_X(-D_{X,\xi,\eta} + XM_X E_{X,\xi,\eta}) &= A^{-1}S_X(\Pi_X + XM_XX^{\top})(-D_{X,\xi,\eta} + XM_X E_{X,\xi,\eta})\\
&= A^{-1}S_X(-D_{X,\xi,\eta} + XM_X E_{X,\xi,\eta})\\
&= -A^{-1}S_X D_{X,\xi,\eta}\\
&= -A^{-1}S_X(\xi\skew(M_X X^{\top}\eta) + \eta\skew(M_XX^{\top}\xi)).
\end{align}
This completes the proof.
\end{proof}
\begin{remark}
In~\eqref{eq:D}, computing each of $\sym(XM_X\xi^{\top})$, $\sym(XM_X \eta^{\top})$, and $\sym(\xi\eta^{\top})$ costs $O(n^2p)$, whereas computing the whole of \eqref{eq:D_2} costs $O(np^2)$.
Therefore, using~\eqref{eq:D_2} reduces the computational cost.
This reduction is particularly effective in the usual case $p \ll n$.
\end{remark}

Owing to this lemma, $\Gamma^{(1)}_X$ in~\eqref{eq:Gamma1} and $\Gamma^{(2)}_X$ in~\eqref{eq:Gamma2} can be further simplified as follows:
\begin{align}
\Gamma^{(1)}_X(\xi, \eta) ={}& \frac{1}{\rho}A^{-1}\Pi_X B_{X,\xi,\eta} - 2\rho XX^{\top}C_{X,\xi,\eta}\\
&+ XJ\sym(\xi^{\top}A\eta) + 2X\sym(X^{\top}A\xi X^{\top}A\eta) - 2A^{-1}\Pi_X A^{-1}C_{X,\xi,\eta} \notag\\
={}& A^{-1}\Pi_X(\rho^{-1}B_{X,\xi,\eta} - 2A^{-1}C_{X,\xi,\eta}) \\
&+2X(-\rho X^{\top}C_{X,\xi,\eta} + \sym(X^{\top}A\xi X^{\top}A\eta)) + XJ\sym(\xi^{\top}A\eta)
\end{align}
and
\begin{align}
\Gamma^{(2)}_{X}(\xi, \eta) ={}& \frac{1}{\rho}A^{-1}S_X^2 B_{X,\xi,\eta} - \rho X \sym(\xi^{\top}\Pi_X \eta) + XJ\sym(\xi^{\top}A\eta) + 2X\sym(X^{\top}A\xi X^{\top}A\eta)\\
& - A^{-1}S_XD_{X,\xi,\eta} + A^{-1}S_X AXJ \sym(\xi^{\top}\Pi_X \eta)\\
={}& A^{-1}S_X(\rho^{-1}S_X B_{X,\xi,\eta}  - D_{X,\xi,\eta} + AXJ\sym(\xi^{\top}\Pi_X \eta)) \\
&+ X(-\rho \sym(\xi^{\top}\Pi_X\eta) + 2\sym(X^{\top}A\xi X^{\top}A\eta)) + XJ\sym(\xi^{\top}A\eta).
\end{align}

We are in a position to compute the Levi-Civita connection $\nabla$ on the indefinite Stiefel manifold $\iSt_{A,J}(p,n)$.
Since $\iSt_{A,J}(p,n)$ is a Riemannian submanifold of $\E$, we have
\begin{equation}
(\nabla_U V)(X) = P_X^{G}((\bar{\nabla}_{\bar{U}} \bar{V})(X)) = P^{G}_X(\D \bar{V}(X)[\xi] + \Gamma_X(\xi,\eta)),
\end{equation}
where $\bar{U}, \bar{V} \in \mathfrak{X}(\E)$ are smooth extensions of $U, V \in \mathfrak{X}(\iSt_{A,J}(p,n))$ to $\E$.
The following lemma provides a detailed description of $P^{G}_X(\Gamma_X(\xi,\eta))$.
\begin{lemma}
\label{lem:PGamma}
For $X \in \iSt_{A,J}(p,n)$ and $\xi, \eta \in T_X \! \iSt_{A,J}(p,n)$, let $\Omega_{X,\xi} \coloneqq X^{\top}A \xi, \Omega_{X, \eta} \coloneqq X^{\top}A\eta \in \Skew(p)$.
Then, it holds that
\begin{align}
P_X^{G^{(1)}}(\Gamma^{(1)}_X(\xi,\eta)) ={}& A^{-1}\Pi_X(\rho^{-1}B_{X,\xi,\eta} - 2A^{-1}C_{X,\xi,\eta}) \\
&+ 2P^{G^{(1)}}_X(X(-\rho X^{\top}C_{X,\xi,\eta}+\sym(\Omega_{X,\xi}\Omega_{X,\eta})))
\end{align}
and
\begin{align}
P_X^{G^{(2)}}(\Gamma^{(2)}_X(\xi,\eta)) ={}& A^{-1}S_X(\rho^{-1}S_X B_{X,\xi,\eta} - D_{X,\xi,\eta} + AXJ\sym(\xi^{\top}\Pi_X\eta))\\
& + P^{G^{(2)}}_X(X(-\rho \sym(\xi^{\top}\Pi_X \eta) + 2\sym(\Omega_{X,\xi}\Omega_{X,\eta}))).
\end{align}
\end{lemma}

\begin{proof}
For $\Sigma \in \Sym(p)$, $P_X^{G^{(i)}}(XJ\Sigma) = XJ\Sigma - XJ\sym(X^{\top}AXJ\Sigma) = XJ\Sigma - XJ\sym(J^2\Sigma) = XJ(\Sigma - \sym(\Sigma)) = {}  0$ holds for $i = 1, 2$.
Furthermore, for any $K \in \R^{n \times p}$, we have $P_X^{G^{(i)}}(A^{-1}\Pi_X K) = A^{-1}\Pi_X K - XJ\sym(X^{\top}\Pi_X K) = A^{-1}\Pi_X K$ and similarly $P_X^{G^{(i)}}(A^{-1}S_X K) = A^{-1}S_X K$ from $X^{\top}\Pi_X = X^{\top}S_X = 0$.
The conclusion is then straightforward.
\end{proof}

\subsection{Riemannian Hessian of a function on the indefinite Stiefel manifold}
Consider a smooth function $f \colon \iSt_{A,J}(p,n) \to \R$ and its smooth extension $\bar{f} \colon \E \to \R$ to the ambient manifold $\E$.
The Riemannian Hessian of $f$ is defined via the Levi-Civita connection $\nabla$ as
\begin{equation}
\Hess f(X)[\xi] \coloneqq \nabla_{\xi}\grad f
\end{equation}
for $X \in \iSt_{A,J}(p,n)$ and $\xi \in T_X\!\iSt_{A,J}(p,n)$.
Noting \eqref{eq:Rgrad}, we define $\bar{V} \in \mathfrak{X}(\E)$ as
\begin{equation}
\bar{V}(X) \coloneqq G_X^{-1}\grad_{\mathrm{E}} \bar{f}(X) - XJ\sym(X^{\top}AG_X^{-1}\grad_{\mathrm{E}} \bar{f}(X)),
\end{equation}
where $\bar{f}$ is a smooth extension of $f$ to $\E$.
Then, $\bar{V}|_{\iSt_{A,J}(p,n)} = \grad f$ holds.
Using the Leibniz rule, the formula $\D \inv (Y)[Z] = -Y^{-1}ZY^{-1}$ for the derivative of $\inv(Y) \coloneqq Y^{-1}$, and the chain rule, we have
\begin{align}
& \D \bar{V}(X)[\xi]\\
={}& -G_X^{-1}\D G(X)[\xi]G_X^{-1}\grad_{\mathrm{E}} \bar{f}(X) + G_X^{-1}\Hess_{\mathrm{E}} \bar{f}(X)[\xi] - \xi J \sym(X^{\top}AG_X^{-1}\grad_{\mathrm{E}} \bar{f}(X)) \\
&- XJ\sym((\xi^{\top}AG_X^{-1} - X^{\top}AG_X^{-1}\D G(X)[\xi]G_X^{-1})\grad_{\mathrm{E}} \bar{f}(X) + X^{\top}AG_X^{-1}\Hess_{\mathrm{E}} \bar{f}(X)[\xi])\\
={}& P_X^{G}(-G_X^{-1}\D G(X)[\xi]G_X^{-1} \grad_{\mathrm{E}} \bar{f}(X) + G_X^{-1}\Hess_{\mathrm{E}} \bar{f}(X)[\xi])\\
&-\xi J\sym(X^{\top}AG_X^{-1}\grad_{\mathrm{E}} \bar{f}(X)) - XJ\sym(\xi^{\top}AG_X^{-1}\grad_{\mathrm{E}} \bar{f}(X)),\label{eq:LastTerm}
\end{align}
where $\Hess_{\mathrm{E}}\bar{f}$ denotes the Euclidean Hessian of $\bar{f}$, i.e., $\Hess_{\mathrm{E}} \bar{f}(X)[\xi] = \D (\grad_{\mathrm{E}} \bar{f})(X)[\xi]$ holds.
Since $P^{G}_X(XJS) = XJS - XJ\sym(X^{\top}AXJS) = 0$ for any $S \in \Sym(p)$, the last term in~\eqref{eq:LastTerm} vanishes after applying $P_X^{G}$.
By using this fact and $(P^{G}_X)^2 = P^{G}_X$, the Riemannian Hessian can be computed as
\begin{align}
\Hess f(X)[\xi] ={}& \nabla_{\xi}\grad f\\
={}& P^{G}_X(\bar{\nabla}_{\xi} \bar{V})\\
={}& P^{G}_X(\D \bar{V}(X)[\xi] + \Gamma_X(\xi, \bar{V}(X)))\\
={}& P^{G}_X(-G_X^{-1} \D G(X)[\xi] G_X^{-1} \grad_{\mathrm{E}} \bar{f}(X) + G_X^{-1}\Hess_{\mathrm{E}} \bar{f}(X)[\xi] \\
&\qquad\quad - \xi J \sym(X^{\top}AG_X^{-1}\grad_{\mathrm{E}} \bar{f}(X)) + \Gamma_X(\xi, \bar{V}(X))).
\end{align}
At this stage, the formula is still not fully explicit for the two metrics because the projected Christoffel term remains to be specialized.

Since we have computed $P^{G^{(i)}}_X(\Gamma^{(i)}_X(\xi,\eta))$ for $i = 1, 2$ in Lemma~\ref{lem:PGamma}, we can further write out a specific expression of the Hessian.
From Proposition~\ref{prop:G_summary} and Lemma~\ref{lem:PGamma}, we obtain the specific formulas as follows.

{
\begin{proposition}
\label{prop:Riemannian_Hessian_formulas}
Let $f\colon\iSt_{A,J}(p,n)\to\R$ be a smooth function and $\bar f\colon\E\to\R$ be a smooth extension of $f$.
Let $X\in\iSt_{A,J}(p,n)$ and $\xi\in T_X\!\iSt_{A,J}(p,n)$.
Let $G_X^{(1)}$ and $G_X^{(2)}$ be the metric matrices defined in~\eqref{eq:G1_2} and~\eqref{eq:G2_2}, respectively.
Let $\Hess^{(i)}f$ denote the Riemannian Hessian of $f$ with respect to the metric induced by $G_X^{(i)}$ for $i\in\{1,2\}$.
We define $M_X\coloneqq(X^\top X)^{-1}$, $\Pi_X\coloneqq I_n-XM_XX^\top$, and $S_X\coloneqq A-AXJX^\top A$.
Furthermore, for $\zeta\in T_X\!\iSt_{A,J}(p,n)$, we set $\Omega_{X,\zeta}\coloneqq X^\top A\zeta\in\Skew(p)$, and for $\eta\in T_X\!\iSt_{A,J}(p,n)$, we set $B_{X,\xi,\eta}\coloneqq\xi X^\top A\eta+\eta X^\top A\xi$.
We also define
$C_{X,\xi,\eta}
\coloneqq
\sym\left(A\sym(XJ\xi^\top)AS_X\right)\eta+
\sym\left(A\sym(XJ\eta^\top)AS_X\right)\xi-
\sym\left(AS_X\sym(\xi\eta^\top)A\right)XJ$ and 
$D_{X,\xi,\eta}
\coloneqq
XM_X\sym(\xi^\top\eta)+
\xi\skew(M_XX^\top\eta)+
\eta\skew(M_XX^\top\xi)$.
Furthermore, for $i\in\{1,2\}$, we define \begin{equation}
\eta^{(i)}
\coloneqq\grad^{G^{(i)}}f(X)
=
{G_X^{(i)}}^{-1}\grad_{\mathrm E}\bar f(X)
-
XJ\sym\left(
X^\top A{G_X^{(i)}}^{-1}\grad_{\mathrm E}\bar f(X)
\right).
\end{equation}
Then, by using the orthogonal projection~\eqref{eq:projection_simplified}, the Riemannian Hessian associated with $G_X^{(1)}$ and $G_X^{(2)}$ are given as follows:
\begin{align}
\Hess^{(1)} f(X)[\xi]={}& A^{-1}\Pi_X(\rho^{-1}B_{X,\xi,\eta^{(1)}} - 2A^{-1}C_{X,\xi,\eta^{(1)}})\\
&+ P^{G^{(1)}}_X(2X(-\rho X^{\top}C_{X,\xi,\eta^{(1)}} + \sym(\Omega_{X,\xi}\Omega_{X,\eta^{(1)}}))\\
&\qquad\qquad -{G^{(1)}_X}^{-1} \D G^{(1)}(X)[\xi] {G^{(1)}_X}^{-1} \grad_{\mathrm{E}} \bar{f}(X) + {G^{(1)}_X}^{-1}\Hess_{\mathrm{E}} \bar{f}(X)[\xi]\notag \\
&\qquad \qquad - \rho\xi J \sym(JX^{\top}\grad_{\mathrm{E}} \bar{f}(X))),\label{eq:Hess_summary1}
\end{align}
\begin{align}
\Hess^{(2)} f(X)[\xi]={}& A^{-1}S_X(\rho^{-1}S_X B_{X,\xi,\eta^{(2)}} - D_{X,\xi,\eta^{(2)}} + AXJ\sym(\xi^{\top}\Pi_X\eta^{(2)}))\\
&+ P^{G^{(2)}}_X(X(-\rho \sym(\xi^{\top}\Pi_X \eta^{(2)}) + 2\sym(\Omega_{X,\xi}\Omega_{X,\eta^{(2)}}))\\
&\qquad\qquad -{G^{(2)}_X}^{-1} \D G^{(2)}(X)[\xi] {G^{(2)}_X}^{-1} \grad_{\mathrm{E}} \bar{f}(X) + {G^{(2)}_X}^{-1}\Hess_{\mathrm{E}} \bar{f}(X)[\xi] \notag\\
&\qquad \qquad - \rho\xi J \sym(JX^{\top}\grad_{\mathrm{E}} \bar{f}(X))).\label{eq:Hess_summary2}
\end{align}
\end{proposition}
}

Here, the first two lines in both formulas~\eqref{eq:Hess_summary1} and~\eqref{eq:Hess_summary2} stem from the projected Christoffel term $P^{G^{(i)}}_X(\Gamma^{(i)}_X(\xi,\eta))$. The subsequent lines have essentially the same structure in both formulas, although they still depend on the chosen metric through $G_X^{(i)}$, $\D G^{(i)}(X)$, and
$P_X^{G^{(i)}}$.

{
\section{Newton-type methods and trust-region implementation}
\label{sec:Newton}

We consider the optimization problem
\begin{equation}
\label{eq:opt}
\min_{X \in \iSt_{A,J}(p,n)} f(X)
\end{equation}
on the indefinite Stiefel manifold $\iSt_{A,J}(p,n)$, where $f \colon \iSt_{A,J}(p,n) \to \R$ is a smooth objective function to be minimized.
The formulas derived in Section~\ref{sec:second-order} allow us to evaluate the Riemannian gradient $\grad f(X)$ and Hessian $\Hess f(X)[\xi]$ for $\xi \in T_X\!\iSt_{A,J}(p,n)$ under both metrics $G_X^{(1)}$ and $G_X^{(2)}$.
In this section, we review how these quantities can be used in Riemannian Newton-type methods, i.e., the Riemannian Newton's method and Riemannian trust-region method.

\subsection{Riemannian Newton's method}
In Riemannian Newton's method~\cite{AbsMahSep2008,boumal2023intromanifolds,altmann2024riemannian} on $\iSt_{A,J}(p,n)$, at the current point $X_k \in \iSt_{A,J}(p,n)$, we solve Newton's equation 
\begin{equation}
\label{eq:Newton_equation}
\Hess f(X_k)[\xi] = -\grad f(X_k)
\end{equation}
for $\xi \in T_{X_k}\!\iSt_{A,J}(p,n)$.
Once an (approximate) solution $\xi_k$ of~\eqref{eq:Newton_equation} is obtained, the next point $X_{k+1} \in \iSt_{A,J}(p,n)$ is computed as $X_{k+1}=R_{X_k}(\xi_k)$, where $R \colon T\iSt_{A,J}(p,n) \to \iSt_{A,J}(p,n)$ is a retraction~\cite{AbsMahSep2008,boumal2023intromanifolds}, i.e., $R$ is a smooth map satisfying $R_X(0)=X$ and $\D R_X(0)=\id_{T_X\!\iSt_{A,J}(p,n)}$.

If $\Hess f(X_k)$ is positive definite with respect to the given Riemannian metric on $\iSt_{A,J}(p,n)$, then \eqref{eq:Newton_equation} can be solved by the linear conjugate gradient (CG) method in the tangent space.
More precisely, with $k$ being fixed, for an initial guess $\xi_0 \in T_{X_k}\!\iSt_{A,J}(p,n)$, we define the initial residual by $r_0 \coloneqq \Hess f(X_k)[\xi_0]+\grad f(X_k)$.
The linear CG iterates belong to the affine Krylov subspaces $\xi_0+\mathcal K_m(\Hess f(X_k),r_0)$, where
\begin{equation}
\label{eq:Krylov_subspace}
\mathcal K_m(\Hess f(X_k), r_0)
\coloneqq
\Span\{r_0, \Hess f(X_k)[r_0], \dots, \Hess f(X_k)^{m-1}[r_0]\}
\subset T_{X_k}\!\iSt_{A,J}(p,n) .
\end{equation}
Inexact Riemannian Newton's method using Krylov subspace techniques and preconditioning for the multivariate eigenvalue problem was studied in~\cite{zhang2010riemannian}.

It is important to distinguish this linear CG method from the nonlinear
Riemannian CG (RCG) method.
In this paper, the former is used only as a linear solver for Newton's equation~\eqref{eq:Newton_equation} in the fixed tangent space $T_{X_k}\!\iSt_{A,J}(p,n)$.
The latter is a first-order Riemannian optimization method for Problem~\eqref{eq:opt} that generates a sequence of points on the manifold.

The use of the linear CG method is appropriate locally when $\Hess f(X_k)$ is positive definite.
However, for a general smooth objective function, the positive definiteness of the Riemannian Hessian cannot be expected at arbitrary nonstationary points.
In particular, if a search direction $\xi \in T_{X_k}\!\iSt_{A,J}(p,n)$ generated by the linear CG method satisfies $\langle \xi,\Hess f(X_k)[\xi]\rangle_{X_k} \leq 0$, then the linear CG method is no longer appropriate to solve Newton's equation.
This motivates the trust-region method.
The algorithm of the linear CG method is a special case of the truncated CG method (Algorithm~\ref{alg:tCG_TR}), which is discussed in the next subsection.

\subsection{Trust-region model and truncated CG}

In this subsection, we review the Riemannian trust-region framework~\cite{absil2007trust,AbsMahSep2008}.
At $X_k \in \iSt_{A,J}(p,n)$, we define the quadratic model
\begin{equation}
\label{eq:TR_model}
m_k(\xi) \coloneqq f(X_k) + \langle \grad f(X_k),\xi\rangle_{X_k} + \frac{1}{2} \langle \xi,\Hess f(X_k)[\xi]\rangle_{X_k}
\end{equation}
for $\xi \in T_{X_k}\!\iSt_{A,J}(p,n)$.
Then, the trust-region subproblem at $X_k$ is
\begin{equation}
\label{eq:TR_subproblem}
\min_{\xi \in T_{X_k}\!\iSt_{A,J}(p,n)} m_k(\xi)
\quad
\text{subject to}
\quad
\|\xi\|_{X_k} \leq \Delta_k,
\end{equation}
where $\Delta_k>0$ is called the trust-region radius and
$\|\xi\|_{X_k}\coloneqq\sqrt{\langle \xi,\xi\rangle_{X_k}}$.
The subproblem~\eqref{eq:TR_subproblem} is solved approximately by the truncated CG method (Algorithm~\ref{alg:tCG_TR}) in $T_{X_k}\!\iSt_{A,J}(p,n)$.
Starting from $\xi_0=0$, the method applies the CG procedure to the model $m_k$, but terminates early if either a negative-curvature direction is detected or the trust-region boundary is reached.
For $\xi,\eta \in T_{X_k}\!\iSt_{A,J}(p,n)$ and $\Delta>0$, let $\tau(\xi,\eta,\Delta)$ denote the positive scalar $\tau$ satisfying $\|\xi+\tau \eta\|_{X_k}=\Delta$.
Specifically, for $a\coloneqq\langle \eta,\eta\rangle_{X_k}$ and $b\coloneqq\langle \xi,\eta\rangle_{X_k}$, we have
\begin{equation}
\label{eq:tau_boundary}
\tau(\xi,\eta,\Delta) = \frac{-b+\sqrt{b^2+a(\Delta^2-\|\xi\|_{X_k}^2)}}{a}.
\end{equation}
\begin{algorithm}[htb]
\caption{Truncated CG method for solving the trust-region subproblem~\eqref{eq:TR_subproblem}}
\label{alg:tCG_TR}
\begin{algorithmic}[1]
\Require $X_k\in\iSt_{A,J}(p,n)$, 
trust-region radius $\Delta_k>0$, tolerance $\varepsilon_{\rm cg}>0$.
\Ensure Approximate solution $\xi_k\in T_{X_k}\!\iSt_{A,J}(p,n)$ of~\eqref{eq:TR_subproblem}.
\State $\xi_0 \coloneqq 0$, $r_0 \coloneqq \grad f(X_k)$, $\eta_0\coloneqq -r_0$.
\For{$l=0,1,2,\ldots$}
    \If{$\|r_l\|_{X_k}< \varepsilon_{\rm cg}$}
        \State \Return $\xi_l$.
    \EndIf
    \State $h_l \coloneqq \Hess f(X_k)[\eta_l]$.
    \State $\kappa_l \coloneqq \langle \eta_l,h_l\rangle_{X_k}$.
    \If{$\kappa_l \leq 0$}
        \State $\tau \coloneqq \tau(\xi_l,\eta_l,\Delta_k)$ (see~\eqref{eq:tau_boundary}).
        \State \Return $\xi_l+\tau \eta_l$.
    \EndIf
    \State $\alpha_l \coloneqq
    \dfrac{\langle r_l,r_l\rangle_{X_k}}{\kappa_l}$.
    \If{$\|\xi_l+\alpha_l \eta_l\|_{X_k}\ge \Delta_k$}
        \State $\tau \coloneqq \tau(\xi_l,\eta_l,\Delta_k)$ (see~\eqref{eq:tau_boundary}).
        \State \Return $\xi_l+\tau \eta_l$.
    \EndIf
    \State $\xi_{l+1}\coloneqq \xi_l+\alpha_l \eta_l$.
    \State $r_{l+1}\coloneqq r_l+\alpha_l h_l$.
    \State $\beta_{l+1}\coloneqq
    \dfrac{\langle r_{l+1},r_{l+1}\rangle_{X_k}}
    {\langle r_l,r_l\rangle_{X_k}}$.
    \State $\eta_{l+1}\coloneqq -r_{l+1}+\beta_{l+1}\eta_l$.
\EndFor
\end{algorithmic}
\end{algorithm}
When $\Hess f(X_k)$ is positive definite and the trust-region boundary is not reached, Algorithm~\ref{alg:tCG_TR} reduces to the ordinary linear CG method for Newton's equation.

Given an approximate solution $\xi_k$ of~\eqref{eq:TR_subproblem}, we compute the trial point $X_k^+ \coloneqq R_{X_k}(\xi_k)$ by using a retraction $R$ on $\iSt_{A,J}(p,n)$.
The quality of this trial step is evaluated by the ratio
\begin{equation}
\label{eq:TR_ratio}
\varrho_k
=
\frac{f(X_k)-f(X_k^+)}
{m_k(0)-m_k(\xi_k)}
=
\frac{f(X_k)-f(R_{X_k}(\xi_k))}
{-\langle \grad f(X_k),\xi_k\rangle_{X_k}
-\frac{1}{2}\langle \xi_k,\Hess f(X_k)[\xi_k]\rangle_{X_k}}.
\end{equation}
If $\varrho_k$ is sufficiently large, the trial point is accepted; otherwise it is rejected.
The trust-region radius is decreased when $\varrho_k$ is small and increased when the model is sufficiently accurate and the computed step reaches the trust-region boundary.

The Riemannian trust-region method is described in Algorithm~\ref{alg:RTR}.
\begin{algorithm}[htb]
\caption{Riemannian trust-region method with truncated CG method}
\label{alg:RTR}
\begin{algorithmic}[1]
\Require Retraction $R$ on $\iSt_{A,J}(p,n)$, initial point $X_0\in\iSt_{A,J}(p,n)$, initial trust-region radius $\Delta_0>0$,
maximum radius $\Delta_{\max}>0$, 
$\gamma_{\rm dec}, \gamma_{\rm inc}, \varrho_{\rm acc}, \varrho_{\rm inc}$ satisfying
$0<\gamma_{\rm dec}<1<\gamma_{\rm inc}$ and $0<\varrho_{\rm acc}<\varrho_{\rm inc}<1$, and stopping tolerance $\varepsilon>0$.
\For{$k=0,1,2,\ldots$}
\If{$\|\grad f(X_k)\|_{X_k}<\varepsilon$}
    \State \Return $X_k$.
\EndIf
\State Approximately solve~\eqref{eq:TR_subproblem} by
    Algorithm~\ref{alg:tCG_TR} and obtain $\xi_k$.
    \State Compute $\varrho_k$ by~\eqref{eq:TR_ratio}.
    \If{$\varrho_k\ge \varrho_{\rm acc}$}
        \State $X_{k+1}\coloneqq R_{X_k}(\xi_k)$.
    \Else
        \State $X_{k+1}\coloneqq X_k$.
    \EndIf
    \If{$\varrho_k<\varrho_{\rm acc}$}
        \State $\Delta_{k+1}\coloneqq \gamma_{\rm dec}\Delta_k$.
    \ElsIf{$\varrho_k>\varrho_{\rm inc}$ and
    $\|\xi_k\|_{X_k}=\Delta_k$}
        \State $\Delta_{k+1}\coloneqq
        \min\{\gamma_{\rm inc}\Delta_k,\Delta_{\max}\}$.
    \Else
        \State $\Delta_{k+1}\coloneqq \Delta_k$.
    \EndIf
\EndFor
\end{algorithmic}
\end{algorithm}
The trust-region framework is particularly important in the present setting since the Hessian is not necessarily positive definite before convergence.
The truncated CG method in Algorithm~\ref{alg:tCG_TR} explicitly handles negative curvature and trust-region boundary events.
In the numerical experiments in Section~\ref{sec:NumericalExperiments}, we use a Riemannian trust-region implementation based on this framework and the formula for the Hessian derived in Section~\ref{sec:second-order}.
We also compare it with the RCG method.
}

{
\section{Numerical experiments}
\label{sec:NumericalExperiments}

In this section, we demonstrate several numerical experiments to observe the numerical behavior of the second-order optimization methods discussed in Section~\ref{sec:Newton}.
The experiments are designed to address the following points.
First, we compare a first-order method with a second-order trust-region method
using the Hessian formulas derived in Section~\ref{sec:second-order}.
Second, we examine the behavior for several problem sizes.
Third, we investigate the influence of eigenvalues of $A$ approaching $0$.
Finally, we numerically examine the spectrum of the Riemannian Hessian at
stationary points.

We consider the trace minimization problem on $\iSt_{A,J}(p,n)$, i.e., Problem~\eqref{eq:opt} with a specific objective function $f(X) \coloneqq \tr(X^{\top}MX)$ for $X \in \iSt_{A,J}(p,n)$, where $A \in \Sym(n)$ is invertible, $J \coloneqq \diag(I_{p_+}, -I_{p_-}) \in \Sym(p)$ ($p = p_+ + p_-$), and $M \in \Sym_{++}(n)$.
This optimization problem on the indefinite Stiefel manifold is closely related to the generalized eigenvalue problem to find eigenvalues $\lambda \in \R$ and the associated eigenvectors $v \neq 0$ satisfying
\begin{equation}
Mv = \lambda Av,
\end{equation}
especially when we aim to compute $p_+$ positive eigenvalues and $p_-$ negative eigenvalues, ordered by increasing absolute value~\cite{van2025generalized}.
We define the smooth extension $\bar{f}$ of $f$ to $\E$ as $\bar{f}(X) \coloneqq \tr(X^{\top} M X)$.
Then, we have
\begin{equation}
\grad_{\mathrm{E}} \bar{f}(X) = 2MX,
\qquad
\Hess_{\mathrm{E}} \bar{f}(X)[\xi] = 2M\xi.
\end{equation}
The Riemannian gradient and Hessian are then computed by the
formulas in Sections~\ref{sec:Review} and~\ref{sec:second-order}.

All experiments were implemented using Manopt 8.0~\cite{boumal2014manopt} and were performed in double-precision floating-point arithmetic on a computer (Apple M1 Max, 64 GB RAM) equipped with MATLAB R2024a.
The nonlinear Riemannian conjugate gradient method is denoted by RCG, and
the Riemannian trust-region method is denoted by RTR.
The RTR method uses the Riemannian Hessian-vector products derived in
Section~\ref{sec:second-order}, and the trust-region subproblems are solved by the truncated CG method.

Unless otherwise stated, the stopping criterion is $\|\grad f(X_k)\|_{X_k} < 10^{-6}$.
The maximum elapsed time was set to $600$ seconds for each solver run.
For each problem size and run, the same matrices $A$ and $M$ and the same initial point $X_0$ were used for both metrics.
For each fixed problem instance and metric, the RCG and RTR methods were also initialized from the same point.
For the RCG method, we used Manopt's default nonlinear conjugate-gradient update and default line-search routine, with the minimum step size set to $10^{-14}$.
In the RTR method, we used the initial trust-region radius $\Delta_0=10^{-2}$ and the maximum trust-region radius $\Delta_{\max}=1$.
For the construction of test instances, we set $J=\diag(I_{p_+},-I_{p_-})$, $p_+=\lceil p/2\rceil$, and $p_-=p-p_+$.
For the default experiments, the eigenvalues of $A$ were chosen in the intervals $[1,5]$ and $[-5,-1]$, with approximately the same number of positive and negative eigenvalues.
The matrix $M$ in the objective function $f$ was generated as $M=Q_M\diag(\mu_1,\mu_2,\ldots,\mu_n)Q_M^{\top}$, where $Q_M$ is a random orthogonal matrix and the eigenvalues $\mu_i$ are logarithmically spaced so that the condition number of $M$ is $10^2$.
Unless otherwise stated, each experimental setting was repeated $20$ times.
In Tables~\ref{tab:scalability} and~\ref{tab:near_singularity}, the column ``succ.'' reports the number of runs that satisfied the stopping criterion.
Iteration counts and elapsed times are reported as mean $\pm$ standard deviation over successful runs only.
Terminal Riemannian gradient norms are reported as the form of median [first quartile, third quartile] over all runs.
A dash indicates that no run was successful.
Since all runs in the large-scale experiment were successful, the results in Table~\ref{tab:large_scale} are reported as mean $\pm$ standard deviation over all runs.

In the numerical experiments, we used the quasi-geodesic retraction proposed in~\cite{van2025generalized}.
For $\xi\in T_X\iSt_{A,J}(p,n)$, we define $\Omega \coloneqq X^{\top}A\xi$ and $K \coloneqq \xi^{\top}A\xi$.
Since $\xi$ is tangent at $X$, the matrix $\Omega$ is skew-symmetric.
Then, the quasi-geodesic retraction $R$ is given by
\begin{equation}
\label{eq:qg_retr}
R_X(\xi)
=
[X,\xi]\exp\Bigg(\begin{bmatrix}
J\Omega & -JK \\
I_p & J\Omega
\end{bmatrix}\Bigg)
\begin{bmatrix}
I_p\\
0
\end{bmatrix}
\exp(-J\Omega).
\end{equation}
Although this retraction involves matrix exponentials, the required
exponentials are only of sizes $2p\times 2p$ and $p\times p$, not
$n\times n$.
Thus, the exponential part is moderate when $p\ll n$.

\subsection{Scalability and robustness}

We first compared the RCG and RTR methods.
The problem sizes are $(n,p)=(10,4),(50,8),(100,10)$, and we set $\rho=1$.
For each size and for each metric $G_X^{(1)}$ and $G_X^{(2)}$, we performed $20$ random runs.
The maximum number of outer iterations is $1000$ for both methods.
The results are summarized in Table~\ref{tab:scalability}.

\begin{table*}[htb]
\centering
\small
\setlength{\tabcolsep}{4pt}
\caption{Scalability experiment for the trace minimization problem.
The stopping tolerance is $10^{-6}$.
The success column reports the number of successful runs out of $20$.
Iteration counts and elapsed times are reported as mean $\pm$ standard deviation over successful runs only.
Terminal gradient norms are reported as median [first quartile, third quartile] over all runs.}
\label{tab:scalability}
{\scriptsize
\begin{tabular}{ccccccc}
\hline
$(n,p)$ & metric & method & succ. & iter. & time [s] & terminal grad. \\
\hline
$(10,4)$ & $G_X^{(1)}$ & RCG
& $0/20$ & -- & --
& $1.197{\times}10^{-1}\,[9.081{\times}10^{-3},\,5.930{\times}10^{-1}]$ \\
$(10,4)$ & $G_X^{(1)}$ & RTR
& $18/20$ & $18.33 \pm 1.85$ & $0.110 \pm 0.101$
& $2.135{\times}10^{-8}\,[1.520{\times}10^{-9},\,3.162{\times}10^{-7}]$ \\
$(10,4)$ & $G_X^{(2)}$ & RCG
& $1/20$ & $90.00$ & $0.036$
& $5.457{\times}10^{-2}\,[6.356{\times}10^{-5},\,3.563{\times}10^{-1}]$ \\
$(10,4)$ & $G_X^{(2)}$ & RTR
& $20/20$ & $17.60 \pm 1.27$ & $0.072 \pm 0.030$
& $2.644{\times}10^{-8}\,[1.977{\times}10^{-9},\,1.323{\times}10^{-7}]$ \\
\hline
$(50,8)$ & $G_X^{(1)}$ & RCG
& $0/20$ & -- & --
& $2.213{\times}10^{-1}\,[5.022{\times}10^{-2},\,5.621{\times}10^{-1}]$ \\
$(50,8)$ & $G_X^{(1)}$ & RTR
& $19/20$ & $25.95 \pm 1.54$ & $0.515 \pm 0.222$
& $6.194{\times}10^{-9}\,[1.449{\times}10^{-10},\,1.316{\times}10^{-7}]$ \\
$(50,8)$ & $G_X^{(2)}$ & RCG
& $0/20$ & -- & --
& $1.791{\times}10^{-4}\,[4.686{\times}10^{-5},\,3.196{\times}10^{-1}]$ \\
$(50,8)$ & $G_X^{(2)}$ & RTR
& $20/20$ & $23.50 \pm 2.09$ & $0.319 \pm 0.178$
& $1.919{\times}10^{-8}\,[4.150{\times}10^{-11},\,7.531{\times}10^{-8}]$ \\
\hline
$(100,10)$ & $G_X^{(1)}$ & RCG
& $0/20$ & -- & --
& $1.478{\times}10^{-1}\,[6.073{\times}10^{-2},\,3.423{\times}10^{-1}]$ \\
$(100,10)$ & $G_X^{(1)}$ & RTR
& $15/20$ & $37.20 \pm 20.29$ & $2.288 \pm 1.616$
& $6.799{\times}10^{-8}\,[7.361{\times}10^{-9},\,1.015{\times}10^{-6}]$ \\
$(100,10)$ & $G_X^{(2)}$ & RCG
& $0/20$ & -- & --
& $9.722{\times}10^{-4}\,[1.099{\times}10^{-4},\,2.588{\times}10^{-1}]$ \\
$(100,10)$ & $G_X^{(2)}$ & RTR
& $20/20$ & $24.65 \pm 3.42$ & $0.959 \pm 0.681$
& $1.022{\times}10^{-8}\,[3.051{\times}10^{-11},\,2.193{\times}10^{-8}]$ \\
\hline
\end{tabular}
}
\end{table*}

The RCG method reached the prescribed tolerance in only one of the $120$ runs across all problem sizes and metrics.
Accordingly, iteration counts and elapsed times to successful termination were not available for most  instances.
Nevertheless, the terminal gradient norms show that the RCG method often made more progress under $G_X^{(2)}$ than under $G_X^{(1)}$.
The RTR method was considerably more robust.
Under $G_X^{(2)}$, all $60$ runs were successful.
Under $G_X^{(1)}$, $52$ of the $60$ runs were successful.
The observed success counts and terminal gradient norms indicate that RTR with $G_X^{(2)}$ was more robust than with $G_X^{(1)}$ for the tested problem instances.

\subsection{Large-scale stress test}

Next, we examined a larger problem with $(n,p)=(500,25)$.
In this large-scale stress test, we focused on $G_X^{(2)}$, which exhibited the most robust behavior in Table~\ref{tab:scalability}.
In this experiment, we used the RTR method only and set the stopping tolerance to $10^{-5}$.
The results over $20$ random runs are shown in Table~\ref{tab:large_scale}.

\begin{table}[htb]
\centering
\small
\caption{Large-scale stress test with $(n,p)=(500,25)$ using $G_X^{(2)}$.
The stopping tolerance is $10^{-5}$.
All values except the number of successful runs are reported as
mean $\pm$ standard deviation over $20$ runs.}
\label{tab:large_scale}
\begin{tabular}{ccccc}
\hline
metric & succ. & iter. & time [s] & final grad. \\
\hline
$G_X^{(2)}$
& $20/20$
& $38.15 \pm 3.90$
& $84.91 \pm 25.33$
& $6.553{\times}10^{-6} \pm 2.720{\times}10^{-6}$ \\
\hline
\end{tabular}
\end{table}
This experiment shows that the RTR method with $G_X^{(2)}$ can also be applied to problems with $n=500$ and $p=25$.
All $20$ runs reached the prescribed tolerance, which further supports the robustness of the metric with $G_X^{(2)}$ observed in the previous subsection.

\subsection{Near-singularity study}
We next studied the influence of eigenvalues of $A$ approaching zero.
In this experiment, we set $(n,p)=(50,8)$.
Since the theory in this paper assumes that $A$ is nonsingular, we used a zero-free grid and imposed a small positive lower bound on the absolute values of the eigenvalues in floating point arithmetic.
Let $\Xi$ be the following zero-free grid:
\begin{equation}
\Xi
=
\left\{
-1,-\frac{4}{5},-\frac{3}{5},-\frac{2}{5},-\frac{1}{5},
\frac{1}{5},\frac{2}{5},\frac{3}{5},\frac{4}{5},1
\right\}.
\end{equation}
For each fixed $q\in\{1,3,5,11,21\}$ and each $\xi\in\Xi$, we defined $\lambda(\xi,q)\coloneqq\sgn(\xi)\max\{5|\xi|^q,10^{-8}\}$.
For each fixed $q$, each of the ten values in $\{\lambda(\xi,q)\mid\xi\in\Xi\}$ was assigned multiplicity five to form the $50$ eigenvalues of $A$.
For each run, we generated a random orthogonal matrix $Q_A$ and set $A=Q_A\diag(\lambda_A)Q_A^\top$.
The five choices of $q$ define five separate experimental settings and are not combined into a single spectrum.
The numerical lower bound $10^{-8}$ affects only the value corresponding to $|\xi|=1/5$ and $q=21$.
In particular, no two distinct grid magnitudes are mapped to the same positive eigenvalue by the lower-bound operation.
The purpose of this experiment is to study the effect of eigenvalues of $A$ approaching zero rather than the separation between distinct eigenvalues of $A$.
For each fixed $q$ and run, the same matrices $A$ and $M$ and the same initial point were used for both metrics.
We set $\rho=1$ and used the RTR method for both metrics.
For each $q$ and each metric, we performed $20$ runs.
The maximum number of outer iterations was set to $500$.
The results are shown in Table~\ref{tab:near_singularity}.

\begin{table*}[htb]
\centering
\small
\setlength{\tabcolsep}{5pt}
\caption{Near-singularity study with $(n,p)=(50,8)$.
The stopping tolerance is $10^{-6}$.
The success column reports the number of successful runs out of $20$.
Iteration counts and elapsed times are reported as mean $\pm$ standard deviation over successful runs only.
Terminal gradient norms are reported as median [first quartile, third quartile] over all runs.}
\label{tab:near_singularity}
\begin{tabular}{cccccc}
\hline
metric & $q$ & succ. & iter. & time & terminal grad. \\
\hline
$G_X^{(1)}$ & $1$
& $19/20$
& $26.42 \pm 2.36$
& $0.996 \pm 1.802$
& $2.208{\times}10^{-8}\,[1.336{\times}10^{-9},\,3.345{\times}10^{-7}]$ \\
$G_X^{(1)}$ & $3$
& $19/20$
& $288.11 \pm 89.97$
& $5.414 \pm 1.576$
& $2.665{\times}10^{-7}\,[1.315{\times}10^{-7},\,4.518{\times}10^{-7}]$ \\
$G_X^{(1)}$ & $5$
& $0/20$
& --
& --
& $9.489{\times}10^{2}\,[1.326{\times}10^{2},\,1.424{\times}10^{3}]$ \\
$G_X^{(1)}$ & $11$
& $0/20$
& --
& --
& $6.052{\times}10^{6}\,[1.175{\times}10^{6},\,1.838{\times}10^{7}]$ \\
$G_X^{(1)}$ & $21$
& $0/20$
& --
& --
& $1.164{\times}10^{7}\,[2.847{\times}10^{6},\,8.332{\times}10^{7}]$ \\
\hline
$G_X^{(2)}$ & $1$
& $20/20$
& $21.80 \pm 1.64$
& $0.306 \pm 0.260$
& $2.185{\times}10^{-8}\,[3.898{\times}10^{-10},\,3.744{\times}10^{-7}]$ \\
$G_X^{(2)}$ & $3$
& $20/20$
& $20.45 \pm 1.47$
& $0.251 \pm 0.105$
& $2.669{\times}10^{-8}\,[3.711{\times}10^{-11},\,1.126{\times}10^{-7}]$ \\
$G_X^{(2)}$ & $5$
& $20/20$
& $19.40 \pm 1.70$
& $0.243 \pm 0.133$
& $2.495{\times}10^{-8}\,[4.436{\times}10^{-9},\,1.100{\times}10^{-7}]$ \\
$G_X^{(2)}$ & $11$
& $20/20$
& $19.85 \pm 1.50$
& $0.241 \pm 0.243$
& $6.374{\times}10^{-9}\,[1.271{\times}10^{-10},\,9.352{\times}10^{-8}]$ \\
$G_X^{(2)}$ & $21$
& $20/20$
& $21.50 \pm 2.31$
& $0.183 \pm 0.090$
& $6.697{\times}10^{-9}\,[1.362{\times}10^{-10},\,5.135{\times}10^{-8}]$ \\
\hline
\end{tabular}
\end{table*}

For $G_X^{(1)}$, $19$ of the $20$ runs were successful for both $q=1$ and $q=3$.
No run was successful for $q=5$, $q=11$, or $q=21$.
In contrast, all $20$ runs were successful under $G_X^{(2)}$ for every tested value of $q$.
The terminal gradient norms for $G_X^{(1)}$ increased markedly for $q\geq5$.
By contrast, the success counts, iteration counts, and terminal gradient norms for $G_X^{(2)}$ remained stable over all tested values of $q$.
These paired results indicate that, for the present trace minimization problem, the RTR method with $G_X^{(2)}$ was substantially more robust as the smallest absolute eigenvalues of $A$ approached zero.
However, note that this numerical observation does not imply a uniform robustness result for arbitrary objective functions.

\subsection{Condition number of the Hessian}
Finally, we numerically compared the condition number of the Riemannian Hessian under the two metrics at common stationary points.
Before proceeding to the numerical experiment, we note that some properties of the eigenvalues of the Hessian can be understood theoretically.
At a critical point, the bilinear form associated with the Riemannian Hessian is independent of the chosen Riemannian metric because the first derivative of the objective function vanishes.
Consequently, the Riemannian Hessian operators associated with different metrics represent the same symmetric bilinear form, and their numbers of negative, zero, and positive eigenvalues coincide by Sylvester's law of inertia.
The trace minimization problem also has a certain symmetry.
Let $Q=\diag(Q_+,Q_-)$ with $Q_+\in\mathrm O(p_+)$ and $Q_-\in\mathrm O(p_-)$.
Then, it follows from $J = \diag(I_{p_+}, -I_{p_-})$ that $Q^\top JQ=J$, $(XQ)^\top A(XQ)=J$, and $f(XQ)=f(X)$.
Therefore, at a stationary point, the tangent directions generated by this right action belong to the kernel of the Riemannian Hessian.
In the present experiment, we set $p_+=p_-=2$, and hence the dimension of the symmetry group is $\dim\mathrm O(2)+\dim\mathrm O(2)=2$.
Thus, the Hessian has a zero eigenvalue with a multiplicity of at least two.

In the following experiment, we numerically checked whether additional degeneracy occurred and compared the positive eigenvalues, which determine the conditioning of the Hessian away from the symmetry directions.
For each of 100 random problem instances, we first computed a high-accuracy numerical stationary point $X_\star$ by the RTR method with $G_X^{(2)}$ and the stopping tolerance $10^{-10}$.
We then evaluated the Riemannian Hessian operators associated with $G_X^{(1)}$ and $G_X^{(2)}$ at the same point $X_\star$.

Specifically, we computed a basis $B_1,\ldots,B_d$ of $T_{X_\star}\iSt_{A,J}(p,n)$ from the null space of the linearized constraint $\xi^\top AX_\star+X_\star^\top A\xi=0$, where $d=\dim \iSt_{A,J}(p,n) = np-p(p+1)/2$.
The null space was computed by the singular value decomposition.
For each metric $G_X^{(i)}$, $i = 1, 2$, we formed the Gram matrix $C^{(i)}$ whose $(\alpha,\beta)$ element is $\langle B_{\alpha},B_{\beta}\rangle_{X_\star}^{(i)} = \tr(B_{\alpha}^{\top}G_{X_{\star}}^{(i)}B_{\beta})$.
We used its Cholesky decomposition to obtain a basis $E_1^{(i)},\ldots,E_d^{(i)}$ that was orthonormal with respect to the corresponding Riemannian metric.
Using this basis, we formed the matrix representation of the Riemannian Hessian by $H_{\alpha\beta}^{(i)}
=
\langle E_\alpha^{(i)},\Hess^{(i)}f(X_\star)[E_\beta^{(i)}]\rangle_{X_\star}^{(i)}$.

An eigenvalue $\lambda_j$ was classified as numerically zero when $|\lambda_j|\leq10^{-8}\max\{1,\max_k|\lambda_k|\}$.
For both metrics and in all $100$ runs, the numbers of negative, zero, and positive eigenvalues were $0$, $2$, and $28$, respectively.
Thus, the two symmetry-induced zero modes were the only numerically detected degeneracies.
We compared the effective spectral condition number~\cite{NabbenVuik2004}, which is defined as
$\kappa_{\mathrm{eff}}
=
\lambda_{\max}^{+}/\lambda_{\min}^{+}$, where $\lambda_{\max}^{+}$ and $\lambda_{\min}^{+}$ denote the largest and smallest positive eigenvalues, respectively.
We denote the effective spectral condition number associated with the metric $G_X^{(i)}$ by $\kappa_{\mathrm{eff}}^{(i)}$.
The maximum final Riemannian gradient norm over all computed stationary points and both metric evaluations was less than $6.9\times10^{-11}$.

\begin{table}[htb]
\centering
\small
\caption{Effective spectral condition numbers of the Riemannian Hessian at common numerical stationary points over $100$ paired runs.
The condition numbers and paired ratios are reported as median [first quartile, third quartile].}
\label{tab:hessian_conditioning}
\begin{tabular}{cc}
\hline
quantity & value \\
\hline
$\kappa_{\mathrm{eff}}^{(1)}$
& $104.20\,[66.04,\,174.81]$ \\
$\kappa_{\mathrm{eff}}^{(2)}$
& $80.22\,[68.56,\,95.84]$ \\
$\kappa_{\mathrm{eff}}^{(2)}/\kappa_{\mathrm{eff}}^{(1)}$
& $0.702\,[0.477,\,1.172]$ \\
$\#\{\kappa_{\mathrm{eff}}^{(2)}<\kappa_{\mathrm{eff}}^{(1)}\}$
& $69/100$ \\
\hline
\end{tabular}
\end{table}
The results are shown in Table~\ref{tab:hessian_conditioning}.
The median effective spectral condition number was smaller for $G_X^{(2)}$ than for $G_X^{(1)}$.
The median paired ratio $\kappa_{\mathrm{eff}}^{(2)}/\kappa_{\mathrm{eff}}^{(1)}$ was $0.702$.
Moreover, $G_X^{(2)}$ produced the smaller effective spectral condition number in $69$ of the $100$ paired comparisons.
The interquartile range of the effective spectral condition numbers was also narrower for $G_X^{(2)}$.
These results indicate that $G_X^{(2)}$ more frequently produced a better-conditioned Riemannian Hessian in this experiment.
However, $G_X^{(1)}$ produced the smaller effective spectral condition number in the remaining $31$ comparisons.
Thus, we note that the present experiment does not establish a uniform ordering between the two metrics for all problem instances or objective functions.
}

\section{Concluding remarks}
\label{sec:conclusion}
In this paper, we investigated the second-order geometry of the indefinite Stiefel manifold.
In particular, with respect to the two types of Riemannian metrics on the manifold proposed in~\cite{van2025generalized}, the Levi-Civita connection was derived for each case.
This enables us to compute the Riemannian Hessian of a smooth function defined on the manifold.
Furthermore, when considering an optimization problem on the indefinite Stiefel manifold, thanks to the Riemannian Hessian of the objective function, we can implement Newton-type second-order methods.
In the procedure of Riemannian Newton's method, Newton's equation, which is a linear equation in a tangent space defined by using the Hessian of the objective function, should be solved.
We discussed solving Newton's equation by the linear CG method on the tangent space when the Hessian is positive definite.
However, since the Hessian may be indefinite away from a solution, we also considered a Riemannian trust-region framework with the truncated CG method.
Finally, we demonstrated numerical experiments on trace minimization problems on the indefinite Stiefel manifold.
The results showed that the trust-region implementation using the derived Hessian is more robust than the nonlinear Riemannian CG method, especially for the metric $G_X^{(2)}$.
The near-singularity experiments and Hessian spectrum computations also indicated favorable numerical behavior of $G_X^{(2)}$ in the tested problems.

\section*{Acknowledgments}
Funding: This work was partly supported by JSPS KAKENHI Grant Numbers JP25K07125, JP25K03082, and JP24K14985.

\bibliographystyle{unsrt}

\bibliography{sato_bib}
\end{document}